\documentclass[11pt]{article}
\usepackage{mymacros}
\usepackage{mathtools}

\newcommand{\EE}{\bm{E}}

\newcommand{\GFF}{\mathrm{GFF}}
\newcommand{\SG}{\mathrm{SG}}
\newcommand{\PPP}{\mathrm{PPP}}

\newcommand{\ZDM}{{\rm Z}}

\newcommand{\scaling}{\sqrt{8\pi}}
\newcommand{\Phiaux}{\Psi_s}
\newcommand{\XGFFms}{X_s^\GFF}
\newcommand{\XGFFmseps}{X_s^{\GFF_\epsilon}}
\newcommand{\extprocSG}{\eta_{r_\epsilon, \epsilon}^\SG}
\newcommand{\extprocGFFms}{\eta_{r_\epsilon, \epsilon}^{\GFF,s}}
\newcommand{\extprocPhis}{\eta_{r_\epsilon, \epsilon}^s}
\newcommand{\Xsc}{X_s^c}

\newcommand{\ZGFFs}{\ZDM_s^\GFF}

\author{Michael Hofstetter \footnote{University of Cambridge, Statistical Laboratory, DPMMS. }}

\title{Extreme local extrema of the sine-Gordon field}

\begin{document}
\maketitle
\begin{abstract}
We prove that for $\beta<6\pi$ the local extremal process of the massive sine-Gordon field on the unit torus in $d=2$ converges to a Poisson point process with random intensity measure $\ZDM^\SG(dx) \otimes e^{-\alpha h}dh$ for some $\alpha>0$.
The proof combines existing methods for the extremal process associated with the Gaussian free field,
which was introduced and studied by Biskup and Louidor, and a strong coupling between the sine-Gordon field and the Gaussian free field.
\end{abstract}

\section{Introduction and main result}
We consider the (continuum) sine-Gordon field on the two-dimensional torus $\Omega=\T^2$
with mass $m>0$ and coupling constants $z \in \R$ and $0<\beta<6\pi$.
Regularised using a lattice of mesh $\epsilon$ such that $1/\epsilon$ is an integer,
its distribution is the probability measure on $\R^{\Omega_\epsilon}=\{\varphi\colon \Omega_\epsilon \to \R\}$ given by
\begin{equation}   
\label{eq:DefinitionSGMeasure}
  \nu^{\SG_\epsilon}(d\varphi) = 
  \frac{1}{Z_\epsilon}
  \exp\qa{-\epsilon^2\sum_{x\in \Omega_\epsilon} \pa{\frac{1}{2}\varphi_x (-\Delta^\epsilon\varphi)_x + \frac12 m^2\varphi_x^2 + 2 z \epsilon^{-\beta/4\pi} \cos(\sqrt{\beta}\varphi_x)}} d\varphi,
\end{equation}
where $\Omega_\epsilon =\T^2 \cap \epsilon\Z^2$ is the discretised unit torus
and $\Delta^\epsilon$ is the discretised Laplacian, i.e.\
$\Delta^\epsilon f(x) = \epsilon^{-2} \sum_{y \sim x} (f(y)-f(x))$
with $x \sim y$ denoting that $x$ and $y$ are neighbours in $\epsilon\Z^2$.

For $z=0$ the periodic terms in \eqref{eq:DefinitionSGMeasure} disappears and the remaining density is that of a multivariate Gaussian random variable with mean $0$ and covariance $(-\Delta^\epsilon + m^2)^{-1}$, also known as the (massive) Gaussian free field (GFF) on $\T^2$.
In what follows we use the notation $\nu^{\GFF_\epsilon}$ for the law on $\R^{\Omega_\epsilon}$ of this special case.

The sine-Gordon model is a well-studied model in Euclidean quantum field theory,
where measures $\nu$ on the set of distributions $S'(\Omega)$ are studied that have a formal density with respect to the continuum Gaussian free field denoted as $\nu^\GFF$, i.e.\
\begin{equation}
\label{eq:continuum-density-eqft-model}
d\nu(\varphi) \propto e^{-\int_{\Omega} V_0(\varphi(x))dx } d \nu^\GFF(\varphi).
\end{equation}
The choices of $V_0$ are restricted by a set of physically motivated axioms, see \cite{MR0161603} and \cite{MR0329492} \cite{MR0376002},
which are satisfied by the periodic interaction $V_0(\varphi_x) = z\cos(\sqrt{\beta} \varphi_x)$ .
Other choices of physical relevance are the quartic interaction with $V_0(\varphi_x)= \lambda \varphi_x^4$, $\lambda>0$ ($\varphi_2^4$-model)
and the exponential interaction with $V_0(\varphi_x)=\lambda \cosh(\sqrt{\beta}\varphi_x)$ (sinh-Gordon model).
Since the density $\nu$ as stated in \eqref{eq:continuum-density-eqft-model} is ill-defined,
a re\-norm\-al\-isa\-tion and regularisation procedure is necessary to give a rigorous construction of this object. 
A common technique is to replace the continuum Gaussian free field $\nu^\GFF$ by a suitable regularisation
and interpret the non-linearity $V_0$ as a Wick-ordering denoted as $:V_0:$.
For a rigorous definition of the Wick-ordering we refer to \cite[Section 8.5]{MR887102}.
If one regularises the continuum free field by the lattice $\Omega_\epsilon$, that is replace it by the discrete GFF denoted as $\nu^{\GFF_\epsilon}$, then this leads precisely to \eqref{eq:DefinitionSGMeasure}.
The continuum object corresponding to the formal definition in \eqref{eq:continuum-density-eqft-model} can then be constructed as weak limit as $\epsilon \to 0$.
Under suitable assumptions on $z$ and $\beta$,
it is known that the probability measures $\nu^{\SG_\epsilon}$ converge weakly as $\epsilon \to 0$ to a non-Gaussian
probability measure $\nu^{\SG}$ on $H^{-\kappa}(\T^2)$ for any $\kappa>0$.

Such a rigorous construction of the sine-Gordon model is non-trivial and was for instance achieved in \cite{MR914427} and \cite{MR1240586}, \cite{MR1777310}.
The sine-Gordon model in $d=2$ on the full space $\R^2$ was constructed in \cite{Barashkov2022SineGordon} using stochastic control techniques, while the work \cite{BauerschmidtWebb2023Coleman} studies the massless theory $m\to 0$ at the free fermion point $\beta=4\pi$.
Moreover, dynamical aspects of the sine-Gordon field were studied in \cite{MR3452276} and \cite{1808.02594}.
For more details on the sine-Gordon model, we refer to the introduction of \cite{MR4399156}. 

Using a coupling between the sine-Gordon field and the Gaussian free field it was recently proved in \cite{MR4399156} that the centred global maximum of the sine-Gordon field converges in distribution to a randomly shifted Gumbel distribution; see Section \ref{sec:SG-max} for the precise statement.
In this article we continue to investigate the extremal behaviour of the sine-Gordon field by capturing also the local maxima and their locations.

Let $\Phi^{\SG_\epsilon}$ be a realisation of the sine-Gordon field on $\Omega_\epsilon$, i.e.\ a random variable with values in $\R^{\Omega_\epsilon}$ and distribution $\nu^{\SG_\epsilon}$.
Following \cite{MR3509015} we call a point $x\in \Omega_\epsilon$ an $r$-local maximum of $\Phi^{\SG_\epsilon}$, if
\begin{equation}
\label{eq:SG-r-local-max}
\Phi^{\SG_\epsilon}(x)= \max_{\Lambda_r(x)}\Phi^{\SG_\epsilon},
\end{equation}
where $\Lambda_r(x)=\{y\in \Omega_\epsilon\colon |x-y|\leq r\}$
and $|\cdot|$ is the Euclidean distance on the continuum torus $\Omega$. 
Let $\Theta_r^\SG$ be the set of $r$-local maxima of $\Phi^{\SG_\epsilon}$. 

Our main object of interest is the local extremal process associated with the sine-Gordon field which encodes both the locations of the local extrema and their values.
It is formally given as the $\epsilon \to 0$ limit of the random measures $\extprocSG$ on $\Omega\times \R$ defined as
\begin{equation}
\label{eq:defn-extremal-process-sg}
\extprocSG\equiv \sum_{x\in \Theta_{r_\epsilon}^\SG} \delta_x\otimes\delta_{\Phi^{\SG_\epsilon}(x) - m_\epsilon},
\end{equation}
where $\delta$ denotes the Dirac measure, $(x, \Phi^{\SG_\epsilon}(x)-m_\epsilon)\in \Omega_\epsilon \times \R$ are random and 
\begin{equation}
\label{eq:m-eps}
m_\epsilon = \frac{1}{\sqrt{2\pi}} (2\log \frac{1}{\epsilon} - \frac{3}{4}\log \log \frac{1}{\epsilon}).
\end{equation}
Throughout this work, we refer to $\extprocSG$ and related objects as extremal process.
However, we emphasise that we do not attempt to study the \textit{full} extremal process,
which includes all local maxima as well as the shape of the field around each local maxima.

Studying the limiting behaviour of \eqref{eq:defn-extremal-process-sg} when $\epsilon \to 0$,
one has to impose conditions on the sequence $(r_\epsilon)_\epsilon$, so that all extreme local maxima of the field are taken into account.
The choice of this sequence is motivated by the typical shape of the extremal landscape of the field; see the remark below Theorem \ref{thm:coupling-bd}.
The main result of this article is the following convergence of $\extprocSG$ to a Poisson point process (PPP) with random intensity measure,
where convergence here and throughout this paper is understood with respect to the topology of vague convergence on the space of Radon measures on $\Omega\times \R$.

\begin{theorem}
\label{thm:convergence-to-ppp}
There is a random measure $\ZDM^{\SG}(dx)$ on $\Omega$ with $\ZDM^\SG(\Omega) <\infty$ a.s.\ and $\ZDM^\SG(A)>0$ a.s.\ for every non-empty open set $A\subseteq \Omega$,
such that for any sequence $(r_\epsilon)_\epsilon$ with $r_\epsilon \to 0$ and $r_\epsilon/\epsilon \to \infty$,
\begin{equation}
\label{eq:convergence-to-PPP-SG}
\extprocSG \to \PPP(\ZDM^{\SG}(dx) \otimes e^{-\scaling h}dh).
\end{equation}
\end{theorem}

A Poisson point process with random intensity measure $\lambda$, here denoted as $\PPP(\lambda)$, is sometimes referred to as a Cox-Process.
Such a process can be sampled by first sampling from the distribution of $\lambda$,
and then sampling independently a usual Poisson point process with deterministic intensity measure given by the sample of $\lambda$. 

Theorem \ref{thm:convergence-to-ppp} is known for the by now well-understood massless Gaussian free field (GFF) with Dirichlet boundary condition; see \cite[Theorem 1.1]{MR3509015}.
The interpretation of the limiting Poisson process is that in the limit the heights of the local extrema become independent,
while their spatial correlation is described by a random measure.
It is remarkable that such a precise understanding is also obtained for the non-Gaussian sine-Gordon field.
In fact, it is conjectured that analogous results hold for other log-correlated fields and that the limiting point process comes from the small scales of the field; see \cite{PhysRevE.63.026110}.
However, few results confirming this universality have been rigorously established so far.

The minor difference between the constant $\scaling$ in \eqref{eq:convergence-to-PPP-SG} and the analogous result for the GFF comes from the different scaling:
here the normalisation of the field is such that $\var{\Phi^{\SG_\epsilon}(0)} \sim \frac{1}{2\pi} \log \frac{1}{\epsilon}$. 
We also emphasise that we consider the discretised fields on the discretised unit torus $\Omega_\epsilon\subseteq \epsilon \Z^2$ rather than on increasing boxes $(0,N)^2 \cap \Z^2$ as in \cite{MR3509015}. 
In order to correctly scale the sequence $(r_N)_N$ from $(0,N)^2\cap \Z^2$ to our setting,
one has to multiply by $1/N$, which results in the seemingly different conditions for the sequence $(r_\epsilon)_\epsilon$ in Theorem \ref{thm:convergence-to-ppp} compared to \cite[Theorem 1.1]{MR3509015}.

Finally, we remark that Theorem \ref{thm:convergence-to-ppp} extends \cite[Theorem 1.1]{MR4399156} on the convergence of the centred global maximum of the sine-Gordon field.
Indeed, the distribution function of $\max_{\Omega_\epsilon}\Phi^{\SG_\epsilon} -m_\epsilon$ at $x\in\R$ can be expressed in terms of $\extprocSG$ as the probability that no points are seen above height $x\in\R$.
The convergence of $(\extprocSG)_\epsilon$ to a Poisson point process then implies that the distribution function of the centred maximum converges and that the limit is given by the avoidance probability of the set $\Omega \times (x,\infty)$.
By a standard result for Poisson point process this probability is equal to the Laplace transform of $\ZDM^\SG(\Omega)$ with spectral parameter $\frac{1}{\sqrt{8\pi}} e^{-\sqrt{8\pi}}$,
where $\ZDM^\SG(\Omega)$ is the total mass of the measure $\ZDM^\SG(dx)$ from Theorem \ref{thm:convergence-to-ppp}.
Comparing this with \cite[Theorem 1.1]{MR4399156}, it follows that the random shift therein, denoted as $\ZDM^\SG$, is in fact equal to $\ZDM^\SG(\Omega)$.
The main difference between the present work and the discussion of the maximum in \cite[Section 4]{MR4399156} is that the subdivision of the torus into boxes of macroscopic side length $1/K$ used in \cite{MR4399156} is not needed here.
Instead, we use the convergence of Laplace functionals to identify limiting objects as Poisson point point processes and,
using the coupling between the GFF and the sine-Gordon field,
show that the key results in \cite{MR3509015} continue to hold when adding a continuous non-Gaussian and independent field.
Difficulties arise when considering the joint convergence of the extremal process associated with the GFF and the continuous field, since both objects are not independent.

\section{Existing results and related work}
\label{sec:existing-results}
In this section we give a detailed literature review on results that are particularly relevant to the present article. 
The first part is devoted to extreme values for log-correlated Gaussian fields, while the second part focuses on results related to the sine-Gordon field.

\subsection{Extreme values of log-correlated fields}
In recent years there has been significant progress in the study of extreme values of log-correlated fields, in particular in that of the Gaussian free field.
Generally speaking the main interest in this area is on understanding the limiting behaviour of extremal field values as the regularising parameter $\epsilon>0$ tends to $0$.
Since the limiting object of the random field takes values in the space of distributions, it is clear that the local and global extremes diverge when $\epsilon \to 0$,
and hence, a recentring of the extremal field values has to be applied before addressing further questions.
We remark that most results which we mention below were proved for fields defined on a boxes $\big([0,N)\cap \Z\big)^2$ with Dirichlet boundary condition and with $N\to \infty$. 
As for the space regularisation, this approach is completely equivalent to our lattice regularisation using a grid of width $\epsilon>0$.
Thus, to be consistent in the notation, we state all results in this section for fields defined on $\big( [0,1)\cap \epsilon \Z\big)^2$ unless specified otherwise.

For the GFF the correct recentring was found to be $m_\epsilon$ as defined in \eqref{eq:m-eps},
which equals the expectation value of the global maximum of the GFF up to order $1$, i.e.\ for $\Phi^{\GFF_\epsilon} \sim \nu^{\GFF_\epsilon}$, as $\epsilon \to 0$,
\begin{equation}
\E[\max_{\Omega_\epsilon} \Phi^{\GFF_\epsilon} ] = m_\epsilon + O(1).
\end{equation}
The first order logarithmic term was identified by Bolthausen, Deuschel and Giacomin in \cite{MR1880237} using a multiscale analysis of the maximum to reduce the problem to a field with underlying tree structure.
In particular, the first order divergence is the same as if the field values were independent,
i.e.\ if $(\Phi^\epsilon(x))_{x\in \Omega_\epsilon}$ was a collection of independent centred Gaussian random variables with variances $\var(\Phi^\epsilon(x)) = \frac{1}{2\pi}\log \frac{1}{\epsilon}$.
This means that even though the GFF is a correlated field, its maximum is in first order the same as if there was no correlation at all.

However, the correlations become apparent in the second order: for the GFF maximum  the loglog-term comes with the prefactor $-\frac{1}{\sqrt{2\pi}}\frac{3}{4}$,
while in the iid case the corresponding prefactor is $-\frac{1}{\sqrt{2\pi}}\frac{1}{4}$.
Thus, the spatial correlations lead to a damping effect, for which the GFF maximum is typically smaller than that of independent random variables.

This correction in the second order was first identified by Bramson in \cite{MR494541} for branching Brownian motion and later for branching random walk by Addario-Berry and Reed in \cite{MR2537549}.
Using a comparison between the maximum of the GFF and the maximum of a (modified) branching random walk
Bramson and Zeitouni proved in \cite{MR2846636} that this correction also holds for the GFF maximum and that the sequence $(\max_{\Omega_\epsilon} \Phi^{\GFF_\epsilon}-m_\epsilon)_\epsilon$ is tight. 

Finer properties of the GFF extremal values were revealed by Ding and Zeitouni in \cite{MR3262484}, where the authors studied the sets of near maxima, i.e.\ random sets of vertices of the form
\begin{equation}
\Gamma_\epsilon^\GFF(\lambda) = \{ x\in \big( [0,1)\cap \epsilon \Z\big)^2 \colon \Phi^{\GFF_\epsilon}(x) \geq m_\epsilon -\lambda \}.
\end{equation}
We state two of their main results here, as they are particularly important later in this article.
The first one is on the geometry of the set of near maxima and states that vertices
where the field attains high values are typically either close to or far away from each other.
\begin{theorem}[\hspace{-3.1pt} {\cite[Theorem 1.1]{MR3262484}}]
\label{thm:near-maxima}
There is a constant $c>0$, such that we have
\begin{equation}
\label{eq:near-maxima}
\lim_{r\to \infty} \limsup_{\epsilon \to 0} \P(\exists u,v \in \Gamma_\epsilon^\GFF (c\log\log r) \colon \epsilon r < |u-v| < 1/r)=0.
\end{equation}
\end{theorem}
This result was later generalised to all log-correlated Gaussian fields with the constant $c>0$ being uniform for all such fields; see \cite[Lemma 3.3]{MR3729618} for the precise statement. 

The second main result in \cite{MR3262484} states that the size of the set of near maxima grows exponentially in $\lambda$.
\begin{theorem}[ \hspace{-3.5pt}{\cite[Theorem 1.2]{MR3262484}}]
\label{thm:size-level-sets}
There are constants $0<c<C$ such that 
\begin{equation}
\lim_{\lambda\to \infty} \liminf_{\epsilon \to 0} \P(ce^{c\lambda } \leq |\Gamma_\epsilon^\GFF(\lambda)| \leq Ce^{C\lambda}) = 1. 
\end{equation}
\end{theorem}
The key here is that the bounds on the size of the set do not depend on the regularisation parameter $\epsilon$.
This allows us to record the following immediate consequence of Theorem \ref{thm:size-level-sets} for the later use. For any fixed $\lambda >0$ we have
\begin{equation}
\label{eq:level-set-size-large}
\lim_{M\to \infty} \limsup_{\epsilon \to 0} \P(|\Gamma_\epsilon^\GFF(\lambda)|>M)=0.
\end{equation}
In this work we invoke these results for massive Gaussian free fields on the unit torus with periodic boundary conditions.
We argue in Section \ref{app:size-level-sets} that the statements continue to hold under these somewhat different assumptions. 

After these results were established, the scope shifted to the distribution of the global maximum and related objects.
In \cite{MR3433630}  Bramson, Ding and Zeitouni established convergence in law for the centred GFF maximum to a randomly shifted Gumbel distribution, i.e.\
\begin{equation}
\label{eq:convergence-GFF-max}
\max_{\Omega_\epsilon} \Phi^{\GFF_\epsilon} - m_\epsilon \to \frac{1}{\sqrt{8\pi}}X + \frac{1}{\sqrt{8\pi}} \log \ZDM^\GFF + b
\end{equation} 
in distribution, where $X$ is a standard Gumbel variable, $\ZDM^\GFF$ is a positive random variable and $b$ is a constant.
As a by-product of their proof, they obtain that the random shift $\ZDM^\GFF$ takes a similar form as the so-called derivative martingale in \cite{MR893913}.
Finally, Ding, Roy and Zeitouni generalised in \cite{MR3729618} the convergence \eqref{eq:convergence-GFF-max} to general Gaussian log-correlated fields.
In fact, it is conjectured that the growth of the maximum and the convergence to a randomly shifted Gumbel distribution are universal for a larger class of log-correlated fields, Gaussian or not.
However, since many existing methods rely on the Gaussian nature of the field, only few results are available in the non-Gaussian regime; see for instance \cite{MR3933043}, \cite{MR4164451} and \cite{MR4399156}. 
There is also a substantial literature for the non-Gaussian field of the logarithm of the characteristic polynomial of random matrices; see for example \cite{MR3848227}, \cite{MR3848391} and \cite{MR3594368}.
The surprising relation between the extrema of characteristic polynomials
and log-correlated processes has been pointed out and explored for the first time in \cite{108.170601} and \cite{MR3151088}, thus laying the foundation for subsequent activity in this direction.

Parallel to the progress on the global maximum of the log-correlated Gaussian fields Biskup and Louidor studied the extremal process of the massless GFF with Dirichlet boundary condition in \cite{MR3509015}
and established convergence to a Poisson point process on $[0,1]^2\times \R$. They found that as $\epsilon\to 0$
\begin{equation}
\label{eq:extremal-gff-convergence-to-ppp}
\eta_{r_\epsilon,\epsilon}^\GFF \to \PPP(\ZDM^\GFF(dx)\otimes e^{-\alpha h}\, dh),
\end{equation}
where $\eta_{r_\epsilon,\epsilon}^\GFF$ is defined analogously to \eqref{eq:defn-extremal-process-sg}, but with $\Phi^{\SG_\epsilon}$ replaced by $\Phi^{\GFF_\epsilon}$,  $\ZDM^\GFF(dx)$ is a random measure on $[0,1]^2$ and $\alpha=\sqrt{8 \pi}$ with our scaling.
A main step of this proof is to show that the limiting point process is invariant under the evolution of its points by independent Brownian motions with drift $-\frac{\alpha}{2} t$.
Interestingly, this invariance property manifests itself in the deterministic Gumbel law $e^{-\alpha h}\, dh$ for the height component.

In this article we need a small generalisation of this convergence result to massive GFFs with periodic boundary condition on the unit torus $\Omega$.
Let $\XGFFmseps$ be the centred Gaussian field on $\Omega_\epsilon$ with covariance $(-\Delta^\epsilon + m^2 + 1/s)$.
The proof of the following statement is completely analogous to the one for the massless GFF with Dirichlet boundary condition given in \cite{MR3509015}. We comment on the essential differences in Appendix \ref{app:generalisation-to-massive-gff}.
\begin{theorem}{(Analogue of \cite[Theorem 1.1]{MR3509015})}
\label{thm:ext-process-massive-GFF}
Let $s\in (0,\infty]$, let $\XGFFmseps$ be the discrete GFF on $\Omega_\epsilon$ with mass $m^2+1/s$ and let $\extprocGFFms$ be the associated extremal process.
Then there is a random measure $\ZGFFs$ on $\Omega$ with $\ZGFFs(\Omega)<\infty$ a.s.\ and $\ZGFFs(A) >0$ a.s.\ for every non-empty open $A\subseteq \Omega$,
such that for any sequence $(r_\epsilon)_\epsilon$ with $r_\epsilon\to 0$ and $r_\epsilon/\epsilon \to \infty$,
\begin{equation}
\extprocGFFms \to \PPP(\ZGFFs(dx)\otimes e^{-\scaling h}dh).
\end{equation}
\end{theorem}

The intensity measure $\ZDM^\GFF(dx)$ in \eqref{eq:extremal-gff-convergence-to-ppp} has also drawn significant attention. 
By verifying that it satisfies a set of characterising properties, Biskup and Louidor showed in \cite{MR4082182} that it is in fact 
(up to a multiplicative constant)
equal to the critical Gaussian multiplicative chaos associated with the field, i.e.\ the $\epsilon \to 0$ limit of the random measures
\begin{equation}
\mu^\epsilon(A) = \epsilon^2 \sum_{x\in A}(\frac{2}{\sqrt{2\pi}} \log \frac{1}{\epsilon} - \Phi^{\GFF_\epsilon}(x) )e^{-2\log \frac{1}{\epsilon} + \scaling \Phi^{\GFF_\epsilon}(x)} \qquad A\subseteq [0,1]^2,
\end{equation}
where again the prefactors differ compared to other references due to our scaling of the field.
Even though it is widely believed that the random shift of the Gumbel distribution or more generally the random part of the intensity measure of the Poisson point process are given by the critical multiplicative chaos associated with the field,
a rigorous proof only exists for the massless GFF thanks to its conformal invariance property.
A thorough discussion of the critical Gaussian multiplicative chaos and its basic properties can be found in \cite{MR3262492} and \cite{MR3395460}. Recent developments in this area are collected in the review \cite{MR4396197} by Powell.

Following their work on the extremal process Biskup and Louidor went on to discuss the full extremal process of the massless GFF, which also takes into account the shape of the field around local maxima. 
Since each high value of the field typically comes with an entire cluster of comparable values,  
a natural extension of the local extremal process is
\begin{equation}
\bar \eta_{r_\epsilon, \epsilon}^\GFF = \sum_{x\in \Theta_{r_\epsilon}^\GFF} \delta_x\otimes\delta_{\Phi^{\GFF_\epsilon}(x) - m_\epsilon} \otimes \delta_{\{\Phi^{\GFF_\epsilon}(x) - \Phi^{\GFF_\epsilon} (x+\epsilon z) \colon z \in \Z^2\}},
\end{equation}  
which is a Radon measure on $[0,1]^2 \times \R \times \R^{\Z^2}$. 
In \cite{MR3787554} they proved that there is a probability measure $\nu$ on $[0,\infty)^{\Z^2}$,
such that under the same assumptions on the sequence $r_\epsilon$,
\begin{equation}
\label{eq:full-extremal-convergence}
\bar \eta_{r_\epsilon,\epsilon}^\GFF \to \PPP(\ZDM^\GFF(dx) \otimes e^{-\alpha h} \,dh\otimes \nu (d \phi))
\end{equation}
as $\epsilon \to 0$ with respect to the topology of vague convergence on the space of point measures on $[0,1]^2\times \R \times \R^{\Z^2}$.
This means that the correlation of the spatial positions of local maxima is described by the measure $\ZDM^\GFF(dx)$,
while in the limit the configurations around each local maximum are independent samples from $\nu$. 
The limiting process in \eqref{eq:full-extremal-convergence} is therefore often referred to as a randomly shifted Gumbel process decorated by independent and identically distributed clusters.
A similar result is also known for Branching Brownian motion (BBM), see the works of Arguin, Bovier and Kistler \cite{MR3129797}, A{\"\i}d{\'e}kon, Berestycki, Brunet and Shi \cite{MR3101852}, and Bovier and Hartung \cite{MR3164771},
but not for other log-correlated fields.
It would be an interesting problem to understand how the measure $\nu$ depends on the underlying spatial structure.
A collection of more open question in this area can be found in \cite[Section 2.5]{MR3787554}.

\subsection{Construction of the sine-Gordon field and coupling with the GFF}
\label{sec:coupling}
In this section we recall the coupling between the sine-Gordon field and the Gaussian free field from \cite{MR4399156}.
This coupling provides an important tool to the analysis of the sine-Gordon field,
which in particular allows to transfer results from the extremal process associated with the Gaussian free field to the one for the sine-Gordon field.

We first recall the notion of the decomposed GFF on $\Omega_\epsilon$, which is defined as follows:
let $(W^\epsilon(x))_{x\in\Omega_\epsilon}$ be independent Brownian motions $W^\epsilon(x)=(W^\epsilon_t(x))_{t\geq 0}$
with quadratic variations $t/\epsilon^2$.
In other words, $W^\epsilon$ is a standard Brownian motion with values in
$\R^{\Omega_\epsilon}$ equipped with the inner product 
\begin{equation} \label{e:innerprod}
\avg{f,g} \equiv \avg{f,g}_{\Omega_\epsilon} \equiv \epsilon^2 \sum_{x \in \Omega_\epsilon} f(x)\bar g(x).
\end{equation}

Moreover, we define the discrete heat kernel $\dot c_t^\epsilon$  on $\Omega_\epsilon$ with periodic boundary condition by
\begin{equation}
\label{e:dotcteps}
\dot c_t^\epsilon = e^{t \Delta^\epsilon} e^{-m^2 t},
\qquad
c_t^\epsilon = \int_0^t \dot c_s^\epsilon \, ds,
\end{equation}
where $\Delta^\epsilon$ is the discretised Laplacian acting on $\R^{\Omega_\epsilon}$ as defined below \eqref{eq:DefinitionSGMeasure}.
The relation to its equivalent with Dirichlet boundary condition and its continuum version are summarised in \cite[Appendix]{MR4399156}.
Then we define the decomposed  GFF by
\begin{equation}
\label{e:GFFepsdef}
\Phi_t^{\GFF_\epsilon} = \int_t^\infty q^\epsilon_{u}\, dW_{u}^\epsilon,
\qquad \text{where }
q_t^\epsilon
= \dot c_{t/2}^\epsilon.
\end{equation}
In particular, $\Phi_0^{\GFF_\epsilon}$ is a realisation of the massive Gaussian free field on $\Omega_\epsilon$, i.e.\ a Gaussian field with covariance
$\int_0^\infty \dot c_t^{\epsilon}\, dt = (-\Delta^\epsilon+m^2)^{-1}$.
The parameter $t$ can be thought of as a scale parameter parameter: since the heat kernel $\dot c_t^\epsilon(x,y)$ is essentially supported on $x,y \in \Omega_\epsilon$ with $|x-y|$ of order $L_t \coloneqq \sqrt{t} \wedge 1/m$, the field $\Phi_t^\GFF$ is typically smooth on scale $L_t$.
The construction as stochastic integral with respect to independent Brownian motions then naturally results in in\-de\-pen\-dence between small scales $\Phi_0^{\GFF_\epsilon}-\Phi_t^{\GFF_\epsilon}$ and large scales $\Phi_t^{\GFF_\epsilon}$.

A similar decomposition can be obtained for the sine-Gordon field from a stochastic differential equation; see Theorem \ref{thm:coupling} below.
To this end, we define for $z\in\R$ the microscopic potential $v_0^\epsilon$ of the sine-Gordon field for $\varphi = (\varphi(x))_{x\in\Omega_\epsilon} \in \R^{\Omega_\epsilon}$ by
\begin{equation}
v_0^\epsilon(\varphi)
= \epsilon^{2} \sum_{x\in \Omega_\epsilon} 2z\epsilon^{-\beta/4\pi} \cos(\sqrt{\beta}\varphi(x)).
\end{equation}
Then, for $t>0$, we define the renormalised potential $v_t^\epsilon\colon \R^{\Omega_\epsilon} \to \R$ via 
\begin{equation} \label{e:vdef-conv}
  e^{-v_t^\epsilon(\varphi)}
  = \EE_{c_t^\epsilon} \pB{ e^{-v_0^\epsilon(\varphi+\zeta)} },
\end{equation}
where $\EE_{c}$ denotes the expectation of the centred Gaussian measure with covariance $c$ and the expectation is over the field $\zeta$.
Equivalently to \eqref{e:vdef-conv}, $v_t^\epsilon$ is the unique
solution to the Polchinski equation
\begin{equation}
\label{eq:PolchinskiSDE-v}
\partial_t v_t^\epsilon
  = \frac12 \Delta_{\dot c_t^\epsilon} v_t^\epsilon - \frac12 (\nabla v_t^\epsilon)_{\dot c_t^\epsilon}^2
  = \frac12 \epsilon^4\sum_{x,y \in \Omega_\epsilon} \dot c_t^\epsilon(x,y)
\qa{\ddp{^2v_t^\epsilon}{\varphi(x)\partial\varphi(y)}
-
\ddp{v_t^\epsilon}{\varphi(x)}
\ddp{v_t^\epsilon}{\varphi(y)}
}
\end{equation}
with initial condition $v_0^\epsilon$
where the matrix element $\dot c_t^\epsilon(x,y)$ associated to $\dot c_t^\epsilon$ is normalised such that
\begin{equation}
\label{e:dotcteps-sum}
  \dot c_t^\epsilon f(x) = \epsilon^2 \sum_{y\in \Omega_\epsilon} \dot c_t^\epsilon(x,y) f(y).
\end{equation}

Both representations of $v_t^\epsilon$, that as a solution to the Polchinski equation \eqref{eq:PolchinskiSDE-v},
and that in terms of Gaussian convolution \eqref{e:vdef-conv}, are useful and led to a number of remarkable results for the sine-Gordon field. 
In \cite{MR4303014} Bauerschmidt and Bodineau used a multidimensional Fourier representation for $v_t^\epsilon$,
initially developed by Brydges and Kennedy in \cite{MR914427},
to established bounds on the the renormalised potential and its derivatives and prove a log-Sobolev inequality for the continuum sine-Gordon field.
In fact, these results also enter the proofs of the following four theorems; see also \cite{MR4399156}.
Here and henceforth, we assume that all fields $\Phi_t^{\GFF_\epsilon}$ are constructed with the same cylindrical Brownian motion $(W_t)_{t\geq 0}$ on $L^2(\Omega)$,
which can be expressed as
\begin{equation}
\label{eq:cylindrical-BM}
W_t = \sum_{k \in 2\pi \Z^2} e^{ik\cdot (\cdot)} \hat W_t(k),
\end{equation}
where the $\hat W(k)$ are independent complex standard Brownian motions subject to $\hat W(k)=\overline{\hat W(-k)}$ for $k \neq 0$
and $\hat W(0)$ is a real standard Brownian motion,
and where the sum over $k \in 2\pi\Z^2$ converges in $C([0,\infty),H^{-1-\kappa}(\Omega))$ for any $\kappa>0$.
This gives rise to a coupling of all fields $\Phi_t^{\GFF_\epsilon}$ (and therefore also of $\Phi_t^{\SG_\epsilon}$ below) simultaneously for all $t\geq 0$ and all $\epsilon\geq 0$,
which we use from now on. Note that we have for $\epsilon=0$
\begin{equation}
\label{e:GFF0def}
  \Phi_t^{\GFF_0} = \int_t^\infty q_u^0 dW_{u} \equiv
  \sum_{k\in2\pi \Z^2}  e^{ik\cdot(\cdot)}\int_t^\infty \hat q_u^0(k) d\hat W_u(k),
  \qquad 
  q_t^0
  = e^{-t(-\Delta + m^2)/2},
\end{equation}
where $q_u^0(k)$ denotes the $k$-th Fourier coefficient of $ q_u^0$. For more details on this construction, we refer to \cite[Section 3.1]{MR4399156}.

In the following statements, $C_0([0,\infty), C(\Omega))$ denotes the space of continuous processes with values in the space of continuous functions $C(\Omega)$ that vanish at infinity. Moreover, we denote by $\Omega_\epsilon^*= \{k\in 2\pi \Z^2\colon -\pi/\epsilon < k_i\leq \pi/\epsilon\}$ and $\Omega^*=2\pi \Z^2$ the Fourier dual spaces of $\Omega_\epsilon$ and $\Omega$.
\begin{theorem} 
\label{thm:coupling}
Let $(\cF^t)_{t\geq 0}$ be the filtration generated by the Brownian motion $(W_t)_{t\geq 0}$ in \eqref{eq:cylindrical-BM}.
  Then, for $\epsilon>0$, there is a unique $\cF^t$-adapted process $\Phi^{\SG_\epsilon} \in C([0,\infty), \R^{\Omega_\epsilon})$ such that for any $t\geq 0$
\begin{equation}
\label{e:Phi-SG-eps-coupling}
\Phi_{t}^{\SG_\epsilon}
    = - \int_t^\infty \dot c^\epsilon_u \nabla v_{u}^\epsilon(\Phi_u^{\SG_\epsilon}) \, du
+ \Phi_t^{\GFF_\epsilon}.
\end{equation}
Analogously, there is a unique $\cF^t$-adapted process $\Phi^{\SG_0}$
with $\Phi^{\SG_0}-\Phi^{\GFF_0}\in C_0([0,\infty), C(\Omega))$ such that
\begin{equation}
\label{e:Phi-SG-0-coupling}
\Phi_{t}^{\SG_0} = - \int_t^\infty \dot c^0_u \nabla v_{u}^0(\Phi_u^{\SG_0}) \, du
+ \Phi_t^{\GFF_0}.
\end{equation}  
In particular, for any $t>0$, $\Phi^{\GFF_\epsilon}_0-\Phi_t^{\GFF_\epsilon}$ is independent of $\Phi_t^{\SG_\epsilon}$ for both $\epsilon>0$ and $\epsilon=0$.
\end{theorem}

We refer to the process $\Phi^{\SG_\epsilon}$ as the decomposed sine-Gordon field. Using It\^o's formula and the fact that the renormalised potential $v_t^\epsilon$ satisfies \eqref{eq:PolchinskiSDE-v}, one can prove that the distribution of $\Phi_0^{\SG_\epsilon}$ is indeed given by \eqref{eq:DefinitionSGMeasure}. 

\begin{theorem}
\label{thm:coupling-sg}
Let $\epsilon>0$. Then $\Phi_0^{\SG_\epsilon}$ is distributed as the sine-Gordon field on $\Omega_\epsilon$ defined in \eqref{eq:DefinitionSGMeasure}.
\end{theorem}

These two results establish a coupling between the sine-Gordon field and the Gaussian free field  simultaneously for all scales $t\geq 0$ and all regularisations $\epsilon > 0$.
This can even be extended to $\epsilon=0$ thanks to the convergence of the gradient of the renormalised potential $\dot c_t^\epsilon \nabla v_t^\epsilon$ as $\epsilon \to 0$;
see \cite[Section 2]{MR4399156} for more details.
Thus, when it is clear from the context, we drop from now on the regularisation parameter $\epsilon$ from the notation.

Note that unlike as for the decomposed GFF, there is no exact independence for the decomposed sine-Gordon field.
Instead, we have independence between the large scales of the sine-Gordon field $\Phi_t^\SG$ and the small scales of the GFF $\Phi_0^\GFF - \Phi_t^\GFF$.  

For $\beta < 6\pi$ the bounds on the gradient of the renormalised potential in \cite{MR4303014} allow to establish remarkable results for the difference field $\Phi_t^{\Delta} \coloneqq \Phi_t^\SG-\Phi_t^\GFF$ for $t\geq 0$, which are collected in the following statements. 

\begin{theorem} \label{thm:Phi-limit}
Under the coupling described above of Theorem~ \ref{thm:coupling}, we have for any $t_0>0$
\begin{equation}
\label{e:Phi-Delta-limit}
\E \sup_{t\geq t_0}\norm{\Phi_t^{\Delta_\epsilon}- \Phi^{\Delta_0}_t}_{L^\infty(\Omega_\epsilon)} \to 0
    \quad \text{as $\epsilon \downarrow 0$.}
\end{equation}
In particular, for any $t\geq 0$, the lattice field $\Phi^{\SG_\epsilon}_t$ converges weakly to $\Phi^{\SG_0}_t$ in $H^{-\kappa}(\Omega)$ as $\epsilon \downarrow 0$ when $\kappa>0$,
where we have identified $\Phi^{\SG_\epsilon}_t$ with the element of $C^\infty(\Omega)$ with the same Fourier coefficients for $k\in \Omega_\epsilon^*$
and vanishing Fourier coefficients for $k\in\Omega^*\setminus\Omega_\epsilon^*$.
\end{theorem}

\begin{theorem}
\label{thm:coupling-bd}
The following estimates hold for every field configuration $\Phi^{\Delta_\epsilon} \equiv \Phi^{\Delta}$ for $\epsilon>0$ or $\epsilon=0$ with all constants deterministic and independent of $\epsilon$.
For any $t\geq 0$, the difference field $\Phi^\Delta$ satisfies the bound
\begin{equation}
\label{e:PhiDelta-bd0t}
\max_x |\Phi^{\Delta}_0(x)-\Phi^{\Delta}_t(x)|
\leq
O_\beta(|z| L_t^{2-\beta/4\pi}), \qquad L_t= \sqrt{t} \wedge \frac{1}{m},
\end{equation}
as well as the following H\"older continuity estimates:
\begin{align}
\max_{x} |\Phi_t^{\Delta}(x)|+ \max_{x} |\partial \Phi_t^{\Delta}(x)| +        
\max_{x,y}\frac{|\partial \Phi_t^\Delta(x)-\partial \Phi_t^\Delta(y)|}{|x-y|^{1-\beta/4\pi}}
&\leq O_\beta(|z|),
&& (0<\beta<4\pi),
\nnb
\label{e:PhiDelta-bd}
\max_{x} |\Phi_t^{\Delta}(x)|+ \max_{x,y}
\frac{|\Phi_t^\Delta(x)-\Phi_t^\Delta(y)|}{|x-y| (1+|\log|x-y||)}
&\leq O_\beta(|z|) ,
&& (\beta=4\pi),
 \\
\max_{x} |\Phi_t^{\Delta}(x)|+ \max_{x,y}
\frac{|\Phi_t^\Delta(x)-\Phi_t^\Delta(y)|}{|x-y|^{2-\beta/4\pi}}
&\leq O_\beta(|z|) ,
&& (4\pi\leq \beta<6\pi).
\nonumber
\end{align}
In addition, for any $t>0$ and $k \in \N$, the following cruder bounds on all derivatives hold:
\begin{equation}
\label{e:delPhiDelta-bd}
L_t^k\|\partial^k\Phi^\Delta_t\|_{L^\infty(\Omega)}
\leq O_{\beta,k}(|z|),
\end{equation}
and in particular $\Phi^\Delta_t \in C^\infty(\Omega)$ for any $t>0$.
\end{theorem}

As an immediate consequence of the coupling between the sine-Gordon field and the Gaussian free field and the uniform $L^\infty$ bound on their difference,
we obtain that Theorem \ref{thm:near-maxima} also applies to the sine-Gordon field,
i.e.\ when $\Gamma_\epsilon^\GFF(\lambda)$ in \eqref{eq:near-maxima} is replaced by $\Gamma_\epsilon^\SG(\lambda)= \{x\in \Omega_\epsilon\colon \Phi^{\SG_\epsilon} \geq m_\epsilon - \lambda\}$.
This further implies that the correct conditions on the sequence $(r_\epsilon)_{\epsilon}$ in \eqref{eq:defn-extremal-process-sg} are as in Theorem \ref{thm:convergence-to-ppp}.

\subsection{Maximum of sine-Gordon field}
\label{sec:SG-max}
Finally, we comment on a recent result that has been obtained by the this author and Bauerschmidt; see \cite[Theorem 1.1]{MR4399156}.
Using the coupling from Section \ref{sec:coupling} and existing results for the GFF maximum it was proved that convergence in law to a randomly shifted Gumbel distribution also holds for the centred sine-Gordon maximum. 

\begin{theorem} 
\label{thm:convergence-to-Gumbel}
Let $0<\beta<6\pi$ and $z\in \R$.
Then the centred maximum of the $\epsilon$-regularised sine-Gordon field $\Phi^{\SG_\epsilon} \sim \nu^{\SG_\epsilon}$ converges in law to a randomly shifted Gumbel distribution:
\begin{equation}
\label{e:thm-max-convergence}
\max_{\Omega_\epsilon} \Phi^{\SG_\epsilon}
    - \frac{1}{\sqrt{2\pi}} \pa {2\log \frac{1}{\epsilon} - \frac34 \log \log \frac{1}{\epsilon}+b}
\to \frac{1}{\sqrt{8\pi}}X + \frac{1}{\sqrt{8\pi}}\log \ZDM^{\SG},
\end{equation}
where $\ZDM^\SG$ is a non-trivial positive random variable (depending on $\beta$ and $z$),
$X$ is an independent standard Gumbel random variable, and $b$ is a deterministic constant.
\end{theorem}

This result could only be established for $\beta < 6\pi$ for the lack of uniform estimates on the remainder term $\Phi_0^\Delta$ beyond that threshold; see Theorem \ref{thm:Phi-limit}.
For the same reason the main result of this article is stated under the same assumption on $\beta$.
It would be interesting to see, if similar bounds could be established for $\beta\geq 6\pi$ in order to extend Theorem \ref{thm:convergence-to-Gumbel} to the full range of $\beta$.

\section{Proof of the main theorem}
\label{sec:proof-main-theorem}
In this section we present the proof of Theorem \ref{thm:convergence-to-ppp} in full detail. 
The strategy is to use the coupling between the sine-Gordon field and the GFF and the existing results on the extremal process associated with the GFF in \cite{MR3509015}.
Since the difference term $\Phi_0^\Delta$ is not independent from  $\Phi_0^\GFF$,
it is a priori not clear if the extremal process associated with the GFF conditional on $\Phi_0^\Delta$ also converges to a Poisson point process.
In order to resolve this inconvenience, we use similar arguments as in \cite[Section 4]{MR4399156} and pass to a new field $\tilde \Phi_s^\SG$ by removing the scales of $\Phi^\Delta$ up to $s>0$.
This new field approximates the original sine-Gordon field sufficiently well and moreover admits an independent decomposition into a (massive) GFF and a well behaved continuous field.
To identify the limit of the extremal process $\eta_{r_\epsilon,\epsilon}^s$ associated with $\tilde \Phi_s^\SG$ as a Poisson point process we calculate the limit of the Laplace functionals
\begin{equation}
\cL_s(f)=\E[e^{-\avg{\eta_{r_\epsilon,\epsilon}^s, f}}], \qquad \avg{\eta_{r_\epsilon,\epsilon}^s, f} = \sum_{x\in \Theta_{r_\epsilon}^s} f(x, \tilde \Phi_s^\SG(x) - m_\epsilon)
\end{equation}
as $\epsilon\to 0$ for $f\colon \Omega\times \R\to [0,\infty)$ continuous and  with compact support, which characterise the limiting point process uniquely; see Proposition \ref{prop:laplace-functional-ppp}.
To this end, we condition on the independent continuous part of $\tilde \Phi_s^\SG$ and use the known convergence of the Laplace functional of the extremal process associated with the (massive) GFF. 
Finally, sending the scale cut-off $s$ to $0$, we recover the extremal process associated with the sine-Gordon field as shown in Lemma \ref{lem:reduction-to-phis}.

\subsection{Independent scale decomposition}
As highlighted below Theorem \ref{thm:coupling-sg}, there is no independence between the large and small scales of the sine-Gordon field. In order to circumvent this, we closely follow the exposition in \cite[Sections 4.1--4.2]{MR4399156} and pass to an auxiliary field $\Phiaux$ for which we can find an independent scale decomposition. Note that by Theorem \ref{thm:coupling} the sine-Gordon field can be written as
\begin{equation}
\Phi_0^\SG = \Phi_0^\GFF - \Phi_s^\GFF + \Phi_s^\SG + R_s, \qquad R_s= -\int_0^s \dot c_t \nabla v_t (\Phi_t^\SG) \, dt.
\end{equation}
Using the independence between the short scales of the GFF and the large scales of the sine-Gordon field, see Theorem \ref{thm:coupling}, it follows that
\begin{equation}
\label{eq:tilde-Phi_SG}
\tilde \Phi_s^\SG \equiv (\Phi_0^\GFF- \Phi_s^\GFF) + \Phi_s^\SG
\end{equation}
admits an independent decomposition for every $s>0$. Moreover, the remainder field $R_s$ satisfies $\max_{\Omega_\epsilon} |R_s| \to 0$ as $s\to 0$ uniformly in $\epsilon$ by Theorem \ref{thm:coupling-bd}.
Therefore, instead of working with the sine-Gordon field directly, we first investigate the extremal process associated with $\tilde \Phi_s^\SG$ and then take the limit $s\to 0$.

In order to connect to the results for the extremal process associated with the Gaussian free field in \cite{MR3509015}, it is useful to replace the short scale field $\Phi_0^\GFF-\Phi_s^\GFF$ in \eqref{eq:tilde-Phi_SG} by a massive GFF.
To this end, let $g_s$ be defined as
\begin{equation}
\label{e:function-g}
g_s(\lambda) = \frac{1}{\lambda}(1-e^{-\lambda s}) - \frac{1}{\lambda + 1/s}
\end{equation}
for $\lambda>0$. Since $g_s>0$, we have in distribution 
\begin{equation}
\Phi^{\GFF}_0-\Phi^{\GFF}_s  \stackrel{d}{=}  \XGFFms + X^h_s,
\label{e:decomp-gaussian-part}
\end{equation}
where the two fields on the right-hand side are independent Gaussian fields with covariances
\begin{align}
\label{eq:GFF-m-s}
\cov(\XGFFms) &= (-\Delta + m^2 + 1/s)^{-1}  \\
\label{eq:Xh}
\cov(X^h_s) &= g_s(-\Delta+m^2).
\end{align}
Note that $\XGFFms$ is a Gaussian free field with mass $m^2 + 1/s$. Moreover, the field $X_s^h$ is a.s.\ H\"older continuous as stated in Lemma \ref{lem:hoelder-expectation} below.
For a proof of this result we refer the reader to \cite[Section 4.1]{MR4399156}.

\begin{lemma} 
\label{lem:hoelder-expectation}
Let $X_s^h$ be a Gaussian field with covariance \eqref{eq:Xh}.
There is $\gamma>0$ such that
\begin{equation}
\label{e:hoelder-expectation}
\sup_{\epsilon \geq 0}\E \Big[  \sup_{x\in\Omega_\epsilon} |X_s^h(x)| + \sup_{x \neq y\in\Omega_\epsilon} \frac{|X_s^h(x)-X_s^h(y)|}{|x-y|^\gamma}  \Big] \leq O_{s,m}(1).
\end{equation}
\end{lemma}
Hence, we may focus on the extremal process $\extprocPhis$ associated with the auxiliary field $\Phiaux$  defined by
\begin{equation}
\label{eq:def-phiaux}
\Phiaux \equiv \XGFFms + \Xsc,
\end{equation}
where the field $\XGFFms$ is as above and
\begin{equation}
\Xsc=X_s^h+\Phi_s^\SG
\end{equation}
is independent of $\XGFFms$ and 
converges to a continuous field as $\epsilon \to 0$. The fields $\Phiaux$ and $\tilde \Phi_s^\SG$ are only equal in distribution for which we choose to use a different notation. 
Note that while the sine-Gordon field $\Phi_0^\SG$ is a non-Gaussian field,
the non-Gaussian part of the auxiliary field $\Phiaux$ is only contained in the continuous field $\Xsc$.
We also emphasise that the field $\Xsc$ is different from the so-called coarse field in \cite{MR3433630} and \cite{MR4399156} in that its regularity does not depend on the box decomposition used in these references.
We therefore refrain from using the notion coarse field, but instead refer to it as \textit{continuous field}.

The following result justifies that with regards to the convergence of the extremal process,
it suffices to consider $\tilde \Phi_s^\SG$ for a fixed $s>0$ instead of $\Phi_0^\SG$ and then take the limit $s\to 0$. 
Its proof is postponed to Section \ref{sec:proof-of-reduction-to-phi-s}.  

\begin{lemma}
\label{lem:reduction-to-phis}
Let $s>0$ and let $\extprocPhis$ be the extremal process associated with $\tilde \Phi_s^\SG$ and assume that for every $s>0$, there is
a random measure $\ZDM_s(dx)$ on $\Omega$ with $\ZDM_s(\Omega)< \infty$, $\ZDM_s(A)>0$ a.s.\ for every non-empty open $A\subseteq \Omega$, such that as $\epsilon\to 0$,
for every continuous function $f\colon \Omega\times \R\to [0,\infty)$ with compact support
\begin{equation}
\label{eq:convergence-ext-proc-Phis-PPP}
\lim_{\epsilon \to 0} \E[e^{-\avg{\extprocPhis,f}}]
= \E\left[\exp\left(-\int(1-e^{-f(x,h)}) \ZDM_s(dx) e^{-\alpha h}dh\right)\right],
\end{equation}
where $\alpha>0$ is the same for all $s$.
Then there is a random measure $\ZDM^\SG(dx)$ on $\Omega$ with $\ZDM^\SG(\Omega)< \infty$, $\ZDM^\SG(A)>0$ a.s.\ for every non-empty open $A\subseteq \Omega$, such that as $s\to 0$
\begin{equation}
\label{eq:weak-convergence-Zs}
\ZDM_s(dx)\to \ZDM^\SG(dx)
\end{equation} 
in distribution and moreover, as $\epsilon\to 0$
\begin{equation}
\label{eq:convergence-extproc-SG}
\extprocSG \to \PPP(\ZDM^\SG(dx) \otimes e^{-\alpha h}dh).
\end{equation}
\end{lemma}

We shall see below invoking standard results on the Laplace functionals of Poisson point processes
that \eqref{eq:convergence-ext-proc-Phis-PPP} actually implies that $\extprocPhis$ converges to a point process $\eta^s\sim \PPP(\ZDM_s(dx) \otimes e^{-\alpha h}dh)$.
Note that this result plays the same role as \cite[Lemma 4.1]{MR4399156} in the proof of convergence in distribution for the centred global maximum of the sine-Gordon field.
The analogy between the random measures $\ZDM_s(dx), \, \ZDM^\SG(dx)$ in Lemma \ref{lem:reduction-to-phis} and positive the random variables $\ZDM_s, \, \ZDM^\SG$ in \cite[Lemma 4.1]{MR4399156} is evident, but the argument for the existence of $\ZDM^\SG(dx)$ is somewhat different.

While the result \cite[Lemma 4.1]{MR4399156} follows rather straightforwardly from the uniform convergence of $R_s$,
the proof here is more delicate due to the lack of control of the event that new local extrema are generated by adding $R_s$.
More precisely, having the definition of $\avg{\extprocPhis,f}$ in mind, one would have to prove that with high probability every local maximum of $\tilde \Phi_s^\SG$ can be mapped to precisely one local maximum of $\Phi_0^\SG$, so that both sums in
\begin{equation}
\label{eq:sums-phis-phi-sg}
|\sum_{\Theta_{r_\epsilon}^s} f(x, \tilde \Phi_s^\SG-m_\epsilon ) - \sum_{\Theta_{r_\epsilon}^\SG} f(x, \Phi_0^\SG-m_\epsilon) \big | >\delta
\end{equation}
contain the same number of points with high probability.
In \eqref{eq:sums-phis-phi-sg} and henceforth, $\Theta_{r_\epsilon}^s$ denotes the set of $r_\epsilon$-local maxima of the field $\tilde \Phi_s^\SG$ or $\Phiaux$.
However, it seems that even the uniform convergence is not strong enough to prove that the event in \eqref{eq:sums-phis-phi-sg} has vanishing probability.
To circumvent this issue, we introduce an additional limit $r\to \infty$ and replace the $r_\epsilon$-local maxima by $r\epsilon$-local maxima.
This intermediate step puts us in the situation to apply Theorem \ref{thm:near-maxima} on the non-occurrence of high points on intermediate scales and then find a one-to-one correspondence between the local maxima of $\tilde \Phi_s^\SG$ and $\Phi_0^\SG$ in a probabilistic sense.
A similar procedure was used in \cite[Section 4.2]{MR3509015} and also appears in Section \ref{sec:proof-of-main-theorem} when these results are generalised to our case.

\subsection{Convergence of the continuous field}
In this section we discuss the convergence of the field $\Xsc=X_s^h+ \Phi_s^\SG$ as $\epsilon \to 0$.
To emphasise the difference between the lattice and the continuum field, we shall use for now $X_s^{c,\epsilon}$ and $X_s^{c,0}$ as well as $X_s^{h,\epsilon}$ and $X_s^{h,0}$.
For the field $\Phi_s^\SG$ we write  $\Phi_s^{\SG_\epsilon}$ and $\Phi_s^{\SG_0}$ in accordance with Theorem \ref{thm:coupling} and \cite[Section 3.2]{MR4399156}, where this convergence is discussed.

To address the convergence we use the following Fourier representation which allows to couple all fields $(X_s^{h,\epsilon})_{\epsilon>0}$ simultaneously.
Let $\Omega^*= 2\pi \Z^2$ and $\Omega_\epsilon^*= \{k\in 2\pi \Z^2\colon -\pi/\epsilon < k_i\leq \pi/\epsilon\}$ be the Fourier duals of $\Omega$ and $\Omega_\epsilon$. 
Let $(X(k))_{k\in 2\pi \Z^2}$ be a collection of independent complex standard Gaussian random variables (all defined on a common probability space)
subject to $X(k) =\overline{ X(-k)}$ for $k \neq 0$ and $X(0)\sim \mathcal{N}(0,1)$.

Recall that $\cov(X_s^{h,\epsilon})=g_s(-\Delta^\epsilon + m^2)\equiv c^\epsilon$, where $g_s$ is as in \eqref{e:function-g}. Since $c^\epsilon\colon \Omega_\epsilon \times \Omega_\epsilon\to \R$ is positive and symmetric,
there is a function $q^\epsilon \colon \Omega_\epsilon \to \R$ such that $[q^\epsilon * q^\epsilon](x-y)=c^\epsilon (x,y)$,
where $*$ denotes the discrete convolution on the lattice $\Omega_\epsilon$.  Let $\hat q^\epsilon(k)$ denote the $k$-th Fourier coefficient of $q^\epsilon$, so that 
\begin{equation}
q^\epsilon(x) = \sum_{k \in \Omega_\epsilon} \hat q^\epsilon (k) e^{ikx}, \qquad x\in \Omega_\epsilon.
\end{equation}
Note that $\hat q^\epsilon(k)= \sqrt{\hat c^\epsilon(k)}$, where $\hat c^\epsilon(k)$ is the Fourier coefficient of $c^\epsilon$ when seen as a function $\Omega_\epsilon \to \R$ via translation invariance. 
Then the random variable $\Phi^\epsilon$ with values in $\R^\Omega$ defined by
\begin{equation}
\label{e:Xsh-eps-fourier}
\Phi^\epsilon(x) = \sum_{k\in \Omega_\epsilon^*} \hat q^\epsilon(k) e^{ikx} X(k)
\end{equation}
and restricted to $x\in \Omega_\epsilon$ is multivariate Gaussian with mean $0$ and covariance $c^\epsilon$. 

In what follows we compare lattice fields with continuum fields by extending functions on $\Omega_\epsilon$ to $\Omega$ through the isometric embedding $I_\epsilon \colon L^2(\Omega_\epsilon) \to L^2(\Omega)$,
where $I_\epsilon f$ has the same Fourier coefficients as $f$ for $k\in \Omega_\epsilon^*$ and vanishing Fourier coefficients for $k\in \Omega^*\setminus \Omega_\epsilon^*$.
When it is clear from the context, we omit $I_\epsilon$ from the notation.

To define the limit $X_s^{h,0}$,
we let $\hat c^0(k)=g_s(|k|^2 + m^2)$ be the $k$-th Fourier coefficient of $c^0$ and set $\hat q^0(k)=\sqrt{\hat c^0(k)}$.
Then the random variable $\Phi^0$ with values in $\R^{\Omega}$ defined by
\begin{equation}
\label{e:Xsh-0-fourier}
\Phi^0(x) = \sum_{k\in \Omega^*} \hat q^0(k) e^{ikx} X(k), \qquad x\in \Omega
\end{equation} 
takes a.s.\ values in the space of continuous functions $C^0(\Omega)$, since $g_s(\lambda) = O(\frac{1}{1+\lambda^2})$.
Moreover, it is centred Gaussian with covariance $c^0=g_s(-\Delta + m^2)$ in the sense that
\begin{equation}
\E[\avg{f,\Phi^0} \avg{g,\Phi^0}]= \avg{c^0 f,g}
\end{equation}
holds for all $f,g\in L^2(\Omega)$.
Hence, we assume from now on that $X_s^{h,\epsilon}$ is constructed by \eqref{e:Xsh-eps-fourier} and \eqref{e:Xsh-0-fourier} for $\epsilon \geq 0$.
Note that all fields $(X_s^{h,\epsilon})_{\epsilon\geq 0}$ are now realised on the same probability space.
The following lemma establishes convergence in $H^\alpha(\Omega)$ for $\alpha<1$ under this coupling as $\epsilon \to 0$.

\begin{lemma}
\label{lem:convergence-Xsh}
Under the coupling described above, we have for every $\alpha<1$ as $\epsilon \to 0$
\begin{align}
\label{eq:convergence-Xsh-l2}
\sup_{x\in  \Omega_\epsilon} \E | X_s^{h,\epsilon}(x) - X_s^{h,0}(x) |^2 &\to 0, \\
\label{eq:convergence-Xsh-sobolev}
\E \|I_\epsilon X_s^{h,\epsilon} - X_s^{h,0} \|_{H^\alpha(\Omega)}^2 &\to 0.
\end{align}
\end{lemma}

\begin{proof}
Let $- \hat \Delta^0(k) = |k|^2$ and $- \hat \Delta^\epsilon(k)= \epsilon^{-2}\sum_{i=1}^2 (2-2\cos(\epsilon k_i))$
be the Fourier multipliers of the continuum and lattice Laplacian, respectively.
For $k \in \Omega_\epsilon^*$, we then have
\begin{equation}
\label{e:fourier-multipliers-difference}
  0 \leq -    \hat \Delta^0(k) + \hat \Delta^\epsilon(k)
  = \sum_{i=1}^2 (k_i^2-\epsilon^{-2}(2-2\cos(\epsilon k_i)))
  \leq |k|^2 h(\epsilon k),
\end{equation}
where $h(x) = \max_{i=1,2} (1 -x_i^{-2}(2-2\cos(x_i)))$ satisfies $h(x) \in [0,1-\kappa]$ with $\kappa = 4/\pi^2$ for $|x|\leq \pi$
and $h(x) =O(|x|^2)$.
In particular, we have for $ k \in \Omega_\epsilon^*$
\begin{equation}
\label{eq:bounds-fourier-multipliers-lattice-laplacian}
\kappa |k|^2 \leq -\hat \Delta^\epsilon (k) \leq |k|^2.
\end{equation}
Moreover, $\hat q^\epsilon(k) = \sqrt{g_s(-\hat \Delta^\epsilon(k) +m^2)}$ for $\epsilon \geq 0$,
and hence, using \eqref{eq:bounds-fourier-multipliers-lattice-laplacian} and the fact that $g_s$ is decreasing
\begin{equation}
\label{eq:bounds-fourier-coefficients-Xsh}
\hat q^0(k) = \sqrt{g_s(|k|^2+m^2)} \leq \hat q^\epsilon(k) \leq \sqrt{g_s(\kappa |k|^2+m^2)}.
\end{equation}

From the Fourier representations \eqref{e:Xsh-eps-fourier} and \eqref{e:Xsh-0-fourier} we have for $x\in \Omega_\epsilon$
\begin{equation}
X_s^{h,\epsilon}(x) - X_s^{h,0}(x) = \sum_{k \in\Omega_\epsilon^*} (\hat q^\epsilon(k) - \hat q^0 (k))e^{ikx}X(k) - \sum_{k \in\Omega^*\setminus \Omega_\epsilon^*} \hat q^0(k) e^{ikx} X(k).
\end{equation}
Squaring this equation, taking expectation and using orthogonality of the complex standard Gaussian random variables $(X(k))_{k\in 2\pi \Z^2}$, we have for $x\in \Omega_\epsilon$
\begin{align}
\label{eq:xsh-variance}
\E |X_s^{h,\epsilon}(x) - X_s^{h,0}(x)|^2  &= \sum_{k \in\Omega_\epsilon^*} |\hat q^\epsilon(k) - \hat q^0 (k)|^2 + \sum_{k \in\Omega^*\setminus \Omega_\epsilon^*} |\hat q^0(k)|^2\nnb
&= \sum_{k\in \Omega^*} |\hat q^\epsilon(k)  \mathbf{1}_{k\in \Omega_\epsilon^*} - \hat q^0 (k) |^2,
\end{align}
where the right hand side is uniform in $x\in \Omega_\epsilon$.
Using \eqref{eq:bounds-fourier-coefficients-Xsh} and the fact that $g_s(\lambda) = O(\frac{1}{1+\lambda^2})$ as $\lambda \to \infty$
we see that the sequence in the summation can be bounded by a summable sequence, and hence,
by the convergence $q^\epsilon(k) \to q^0(k)$ as $\epsilon \to 0$ and the dominated convergence theorem (applied to integration with respect to the counting measure on $\Omega^*$)
we have that the right hand side in \eqref{eq:xsh-variance} vanishes when $\epsilon \to 0$.
This proves \eqref{eq:convergence-Xsh-l2}.

To prove \eqref{eq:convergence-Xsh-sobolev}, we first observe that by the definition of the Sobolev norm, we have
\begin{align}
\|I_\epsilon X_s^{h,\epsilon} - X_s^{h,0}\|_{H^{\alpha}(\Omega)}^2 
= \sum_{k\in \Omega_\epsilon^*} (1+ |k|^2)^{\alpha} \big| ( \hat q^\epsilon(k) - \hat q^0(k) )X(k)\big|^2 
\nnb
+ 
\sum_{k \in \Omega^*\setminus \Omega_\epsilon^*}  (1+ |k|^2)^{\alpha} \big| \hat q^0(k) X(k)\big|^2.
\end{align}
Taking expectation we obtain
\begin{equation}
\label{eq:expectation-sobolev-norm-xsh}
\E\|I_\epsilon X_s^{h,\epsilon} - X_s^{h,0}\|_{H^{\alpha}(\Omega)}^2 
= \sum_{k\in \Omega^*} (1+ |k|^2)^{\alpha} |  \hat q^\epsilon(k) \mathbf{1}_{k\in \Omega_\epsilon^*}  - \hat q^0(k)|^2. 
\end{equation}
Since $|\hat q^0(k)|^2 \leq |\hat q^\epsilon(k)|^2 \leq g_s(\kappa|k|^2+m^2) = O(\frac{1}{1+|k|^4})$,
the sum on the right hand side of \eqref{eq:expectation-sobolev-norm-xsh} is finite for $\alpha<1$ and therefore vanishes as $\epsilon \to 0$ by the dominated convergence theorem similarly as above. 
\end{proof}

A consequence of the convergence in Lemma \ref{lem:convergence-Xsh} is the following result which is needed in the course of the proof of Theorem \ref{thm:convergence-to-ppp}.
It states that at the local extrema of $\XGFFmseps$, we may replace the independent continuous field $X_s^{c,\epsilon}$ by its limit $X_s^{c,0}$.
In what follows $\Theta_{r_\epsilon}^{\GFF}$ denotes the set of $r_\epsilon$-local maxima of the field $\XGFFmseps$.

\begin{proposition}
\label{prop:replace-continuous-with-limit}
Let $s>0$ and let $f\colon \Omega\times \R \to [0,\infty)$ be continuous and with compact support.
Then we have for any $\delta >0$
\begin{equation}
\label{eq:replace-continuous-with-limit}
\lim_{\epsilon \to 0} \P\Big(\big |\sum_{ \Theta_{r_\epsilon}^{\GFF}} f(x, \XGFFmseps(x)  + X_s^{c,\epsilon}(x) - m_\epsilon) - f(x, \XGFFmseps(x) + X_s^{c,0}(x) -m_\epsilon) \big |>\delta \Big) = 0.
\end{equation}
\end{proposition}

\begin{proof}
For $C>0$ and $\epsilon\geq 0$, let $B_{\epsilon,C} = \{ \max_{\Omega_\epsilon} X_s^{c,\epsilon} \leq C\}$. 
Since $f$ has compact support, there is $\lambda_0 >0$ such that $\supp f \subseteq \Omega\times [-\lambda_0, \infty)$. 
Hence, for $\lambda > \lambda_0 + C$,
we can replace on the event $B_{\epsilon, C} \cap B_{0,C}$ the summation over $\Theta_{r_\epsilon}^\GFF$ in \eqref{eq:replace-continuous-with-limit} by a summation over $\Theta_{r_\epsilon}^\GFF \cap \Gamma_\epsilon^\GFF(\lambda)$,
where $\Gamma_\epsilon^\GFF(\lambda) = \{ x\in \Omega_\epsilon\colon \XGFFmseps(x) \geq m_\epsilon -\lambda \}$.
Indeed, for $x \notin \Gamma_\epsilon^\GFF(\lambda)$, we have on $B_{\epsilon, C} \cap B_{0,C}$
\begin{equation}
\XGFFmseps(x) + X_s^{c,\epsilon}(x) - m_\epsilon \leq -\lambda + C < -\lambda_0
\end{equation}
and similarly with $X_s^{c,\epsilon}(x)$ replaced by $X_s^{c,0}(x)$. Thus, we have for $x\notin \Gamma_\epsilon^\GFF(\lambda) \text{ on } B_{\epsilon, C} \cap B_{0,C}$
\begin{equation}
 f(x, \XGFFmseps(x)  + X_s^{c,\epsilon}(x) - m_\epsilon) = f(x, \XGFFmseps(x) + X_s^{c,0}(x) -m_\epsilon) =0.
\end{equation}

Next, define $A_M= \{ |\Gamma_\epsilon^\GFF(\lambda) | \leq M \}$ and 
$A_\kappa = \{ \max_{x\in \Gamma_\epsilon^\GFF(\lambda)} | X_s^{c,\epsilon}(x) - X_s^{c,0}(x) | < \kappa \}$. 
Since $f$ is uniformly continuous, we can choose $\kappa$ small enough such that
\begin{equation}
|h-h'| < \kappa \implies \sup_{x\in \Omega_\epsilon}|f(x,h)-f(x,h')| \leq \delta/M.
\end{equation}
Thus, for such a $\kappa$ we have on the event $A_M \cap A_\kappa$
\begin{equation}
\Big| \sum_{\Theta_{r_\epsilon}^\GFF \cap \Gamma_\epsilon^\GFF(\lambda)} \big( f(x, \XGFFmseps  +X_s^{c,\epsilon} - m_\epsilon) - f(x, \XGFFmseps + X_s^{c,0} -m_\epsilon )\big) \Big| \leq |\Gamma_\epsilon^\GFF(\lambda)|\delta/M \leq \delta 
\end{equation}
which shows that the event in \eqref{eq:replace-continuous-with-limit} is contained in $A_M^c \cup A_\kappa^c \cup B_{\epsilon, C}^c \cup B_{0,C}^c$.

To conclude the proof we show that the probability of $A_M^c \cup A_\kappa^c \cup B_{\epsilon, C}^c \cup B_{0,C}^c$ is arbitrarily close to $0$ in the limit $\epsilon \to 0$.  
To this end, we fix $\tilde \kappa$ and, using Lemma \ref{lem:hoelder-expectation}, \cite[Lemma 4.5]{MR4399156} and Markov's inequality, choose $C$ large enough such that
\begin{equation}
\P(B_{\epsilon, C}^c \cup B_{0,C}^c) < \tilde \kappa.
\end{equation}
Moreover, using \eqref{eq:level-set-size-large}, we choose $M$ large enough such that $\P(A_M^c) \leq \tilde \kappa$ holds uniformly in $\epsilon$. Then we have
\begin{equation}
\label{eq:probability-complements}
\P(A_M^c\cup A_\kappa^c) \leq  \tilde \kappa + \P(A_\kappa^c \cap A_M),
\end{equation}
and hence, it suffices to prove that the probability on the right hand side in \eqref{eq:probability-complements} converges to $0$ as $\epsilon \to 0$. 
To this end, we first observe that
\begin{align}
&\P\big(\max_{x\in \Gamma_\epsilon^\GFF(\lambda)} | X_s^{c,\epsilon}(x) - X_s^{c,0}(x) | > \kappa, \,A_M \big) \leq  \nnb
& \leq \P\big(\max_{x\in \Gamma_\epsilon^\GFF(\lambda)} | X_s^{h,\epsilon}(x) - X_s^{h,0}(x) | > \kappa/2,\, A_M \big) + \P\big(\max_{x\in \Gamma_\epsilon^\GFF(\lambda)} |\Phi_s^{\SG_\epsilon}(x)  - \Phi_s^{\SG_0}(x)| >\kappa/2\big).
\end{align}
In the last display the second term vanishes when $\epsilon \to 0$ by the convergence of the decomposed sine-Gordon field shown in \cite[Theorem 3.1 and Lemma 3.5]{MR4399156}.
For the first term, conditioning on $\Gamma_\epsilon^\GFF(\lambda)$, a union bound and Chebychev's inequality yield
\begin{align}
\P\big(&\max_{x\in \Gamma_\epsilon^\GFF(\lambda)} | X_s^{h,\epsilon}(x) - X_s^{h,0}(x) | > \kappa/2, A_M\big) \nnb
& = \E \Big[ \big( \P(\max_{x\in A} |X_s^{h,\epsilon}(x) - X_s^{h,0}(x)| > \kappa/2) \mathbf{1}_{|A|\leq M}\big)_{A=\Gamma_\epsilon^\GFF(\lambda)} \Big] \nnb
& \leq \E \Big[ \big( |A| \max_{x\in A}\P( |X_s^{h,\epsilon}(x) - X_s^{h,0}(x)| > \kappa/2) \mathbf{1}_{|A|\leq M}\big)_{A=\Gamma_\epsilon^\GFF(\lambda)} \Big] \nnb
& \leq \frac{4}{\kappa^2} \E \Big[ \big( |A| \max_{x\in \Omega_\epsilon} \E |X_s^{h,\epsilon}(x) - X_s^{h,0}(x)|^2  \mathbf{1}_{|A|\leq M}\big)_{A=\Gamma_\epsilon^\GFF(\lambda)}\Big] \nnb
& \leq \frac{4}{\kappa^2} M \max_{x\in\Omega_\epsilon} \E |X_s^{h,\epsilon}(x) - X_s^{h,0}(x)|^2.
\end{align}
Now the expectation converges to $0$ as $\epsilon \to 0$ by Lemma \ref{lem:convergence-Xsh}.
In total, we have proved that the limit of the probability in \eqref{eq:replace-continuous-with-limit} is less than $\tilde \kappa$. Since $\tilde \kappa$ was arbitrary, the statement follows. 
\end{proof}

\subsection{Convergence to a Poisson point process: proof of Theorem \ref{thm:convergence-to-ppp}}
\label{sec:proof-of-main-theorem}
By Lemma \ref{lem:reduction-to-phis} the proof of Theorem \ref{thm:convergence-to-ppp} has been reduced to proving that the extremal process $\extprocPhis$ associated with $\tilde \Phi_s^\SG$ converges in distribution to a Poisson point process for a fixed $s>0$,
which we fix throughout the following sections.
Since $\tilde \Phi_s^\SG \stackrel{d}{=} \Phiaux$,
we may from now on assume that the extremal process $\extprocPhis$ corresponds to the field $\Phiaux$.
To this end, we consider the Laplace functionals 
\begin{equation}
\cL_s(f)=\E[e^{-\avg{\extprocPhis,f}}]
\end{equation}
for $f\colon \Omega\times \R \to \R$ non-negative, continuous and with compact support and prove convergence to the Laplace functional of a Poisson point process.
Once this convergence is established we may invoke the following characterisation for Poisson point processes via their Laplace functionals;
see for instance \cite[Proposition 13.2]{MR3791470}. 

\begin{proposition}
\label{prop:laplace-functional-ppp}
Let $(X,\mathcal{X})$ be a measurable space, let $\ZDM(dx)$ be an a.s\ $\sigma$-finite random measure on $X$, and let $\eta$ be a point process on $X$.
Then $\eta$ is a Poisson point process with intensity measure $\ZDM(dx)$,
if and only if
\begin{equation}
\E[e^{-\avg{\eta,f}}]= \E\Big[\exp\big(-\int_X(1-e^{-f(x)})\ZDM(dx)\big)\Big]
\end{equation}
holds for all measurable functions $f\colon X \to [0,\infty)$.
\end{proposition}

Hence, in light of Lemma \ref{lem:reduction-to-phis}, the main result Theorem \ref{thm:convergence-to-ppp} follows from Theorem \ref{thm:convergence-laplace-functionals-phis}  below,
which is the main task for the remainder of this section. 

\begin{theorem}
\label{thm:convergence-laplace-functionals-phis}
Let $s>0$ and let $f\colon \Omega \times \R\to [0,\infty)$ be continuous and with compact support. Then, as $\epsilon\to 0$,
\begin{equation}
\label{eq:laplace-functional-conv-for-eta-s}
\E[e^{-\avg{\extprocPhis,f}}] \to \E\left[\exp\left(-\int (1-e^{-f(x,h)})\ZDM_s(dx) e^{-\alpha h}dh\right)\right],
\end{equation}
where $\ZDM_s(dx) = e^{\alpha\Xsc} \ZGFFs(dx)$ and $\ZGFFs$ is as in Theorem \ref{thm:ext-process-massive-GFF}.
In particular,
there is a Poisson point process $\eta^s\sim \PPP(\ZDM^s(dx) \otimes e^{-\alpha h}dh)$, such that as $\epsilon \to 0$, 
\begin{equation}
\extprocPhis \to \eta^s.
\end{equation}
\end{theorem}

In the course of the proof of Theorem \ref{thm:convergence-laplace-functionals-phis} we also make use of the following two lemmas
that allow us to replace the sequence $(r_\epsilon)_\epsilon$ by $(r\epsilon)_\epsilon$ by introducing an additional limit $r\to \infty$.
This approach seems to be more complicated at first glance,
but it constitutes a setting in which Theorem \ref{thm:near-maxima} applies.
A similar argument is used in the proof of Lemma \ref{lem:reduction-to-phis} as explained in the comment below \eqref{eq:sums-phis-phi-sg}.
These results are vastly analogous to the ones in \cite[Section 4]{MR3509015}, with the main difference in the continuum field $X_s^c$,
which is non-Gaussian in our case.
We present their proofs in Section \ref{sec:cont-field-at-local-maxima} for completeness.
Recall that $\Phiaux=X_s^\GFF + X_s^c$ and that $\extprocPhis$ is the extremal process associated with $\Phiaux$.
\begin{lemma}[Analogue to {\cite[Lemma 4.4]{MR3509015}}]
\label{lem:r-local-maxima}
Let $f\colon \Omega \times \R \to [0,\infty)$ be measurable and with compact support.
Then, for any sequence $(r_\epsilon)_\epsilon$ satisfying $r_\epsilon/\epsilon\to \infty$ and $r_\epsilon \to 0$, 
\begin{align}
\label{eq:r-local-maxima-1}
\lim_{r\to \infty} \limsup_{\epsilon \to 0}  \P \big(\avg{\eta_{r\epsilon, \epsilon}^s, f} &\neq \avg{\extprocPhis, f} \big) = 0, \\
\label{eq:r-local-maxima-2}
\lim_{r\to \infty} \limsup_{\epsilon \to 0} \P \big(\sum_{\Theta_{r_\epsilon}^\GFF} f(x, \Phiaux - m_\epsilon)  &\neq \sum_{\Theta_{r\epsilon}^\GFF} f(x, \Phiaux - m_\epsilon)  \big) = 0. 
\end{align}
\end{lemma}

The reduction to local maxima in neighbourhoods of linearly shrinking size allows to interchange the sets of local maxima of $\Phiaux$ and $\XGFFms$ in the Laplace functionals of $\extprocPhis$.
Recalling the notation for the set of local maxima, we set $\Theta_{r\epsilon}^s = \{x\in \Omega_\epsilon\colon \Phiaux(x) = \max_{\Lambda_{r\epsilon}(x)} \Phiaux\}$
and similarly $\Theta_{r\epsilon}^\GFF= \{x\in \Omega_\epsilon\colon \XGFFms(x) = \max_{\Lambda_{r\epsilon}(x)} \XGFFms\}$. 

\begin{lemma}
\label{lem:r-limit-laplace-functional}
Let $f\colon \Omega \times \R\to [0,\infty)$ be continuous and with compact support. Then
\begin{equation}
\label{eq:r-limit-laplace-functional}
 \lim_{r\to \infty}\limsup_{\epsilon\to 0} \left | \E[e^{-\sum_{x \in \Theta_{r\epsilon}^s} f(x, \Phiaux-m_\epsilon)}] - \E[e^{-\sum_{x \in \Theta_{r\epsilon}^\GFF} f(x, \Phiaux-m_\epsilon)}] \right | =0.
\end{equation}
\end{lemma}

Using these two results, we are ready to give the proof of Theorem \ref{thm:convergence-laplace-functionals-phis}. 
\begin{proof}[Proof of Theorem \ref{thm:convergence-laplace-functionals-phis}]
For an element $\varphi\in \R^{\Omega}$, let $\tau_\varphi$ denote the translation with respect to $\varphi$, i.e.\
\begin{equation}
\tau_\varphi\colon \R^{\Omega}\to\R^{\Omega}, \Phi \mapsto \Phi + \varphi. 
\end{equation}
To simplify notation we also set
\begin{equation}
\label{eq:notation-shifted-function}
f\circ \tau_\varphi (x,h) = f(x, h+\varphi_x)
\end{equation}
for a function $f\colon\Omega \times \R \to \R$.
Note that if $f\colon \Omega \times \R \to [0,\infty)$ is continuous and with compact support and $\varphi$ is continuous,
then $f\circ \tau_{\varphi} \colon \Omega\times \R \to [0,\infty)$ is continuous and with compact support.

To prove \eqref{eq:laplace-functional-conv-for-eta-s}, we first observe that that by \eqref{eq:r-local-maxima-1} and the fact that $f$ is non-negative we have that
\begin{equation}
\lim_{\epsilon \to 0} \E[e^{-\avg{\extprocPhis, f}} ]  
= \lim_{r\to \infty}\lim_{\epsilon\to 0} \E[e^{-\avg{\eta_{r\epsilon, \epsilon}^s, f}}] = \lim_{r\to \infty}\lim_{\epsilon\to 0} \E[e^{-\sum_{x\in \Theta_{r\epsilon}^s} f(x, \XGFFmseps + X_s^{c,\epsilon}-m_\epsilon)}].
\end{equation}
By Lemma \ref{lem:r-limit-laplace-functional}, the summation in the last display can be switched from $x\in \Theta_{r\epsilon}^s$ to $x\in \Theta_{r\epsilon}^\GFF$, i.e.\
\begin{equation}
\lim_{\epsilon \to 0} \E[e^{-\avg{\extprocPhis, f}} ]   = 
\lim_{r\to \infty}\lim_{\epsilon\to 0} \E[e^{-\sum_{x\in \Theta_{r\epsilon}^\GFF} f(x, \XGFFmseps + X_s^{c,\epsilon}-m_\epsilon)}].
\end{equation}
Using now \eqref{eq:r-local-maxima-2}, we have that
\begin{equation}
\lim_{\epsilon \to 0} \E[e^{-\avg{\extprocPhis, f}} ]   =
\lim_{\epsilon \to 0} \E[e^{-\sum_{x\in \Theta_{r_\epsilon}^\GFF} f(x, \XGFFmseps + X_s^{c,\epsilon}-m_\epsilon)}].
\end{equation}
By Proposition \ref{prop:replace-continuous-with-limit}, we can now replace the field $X_s^{c,\epsilon}$ by its limit $X_s^{c,0}$, which yields
\begin{equation}
\lim_{\epsilon \to 0} \E[e^{-\avg{\extprocPhis, f}} ]   
= \lim_{\epsilon \to 0} \E[e^{-\sum_{x\in \Theta_{r_\epsilon}^\GFF} f(x, \XGFFmseps + X_s^{c,0}-m_\epsilon)}].
\end{equation}
Conditioning on the field $X_s^{c,0}$ and using the notation \eqref{eq:notation-shifted-function}, this leads to 
\begin{align}
\label{eq:calculation-laplace-functionals-phis}
\lim_{\epsilon \to 0} \E[e^{-\avg{\extprocPhis, f}} ]  
&= \lim_{\epsilon \to 0} \E[e^{-\sum_{x \in \Theta_{r_\epsilon}^\GFF} f\circ \tau_{X_s^{c,0}} (x, \XGFFmseps-m_\epsilon)}] = \nnb
& =\lim_{\epsilon \to 0} \E \Big[ \E[e^{-\sum_{x \in \Theta_{r_\epsilon}^\GFF} f\circ \tau_{\varphi} (x, \XGFFmseps-m_\epsilon)}] \bigm|_{\varphi=X_s^{c,0}} \Big] \nnb
&=\lim_{\epsilon \to 0} \E \Big [ \E[  e^{-\avg{\extprocGFFms, f\circ \tau_{\varphi}}} ] \bigm|_{\varphi=X_s^{c,0}} \Big ].
\end{align}
Note that the field $X_s^{c,0}$ is bounded and continuous a.s.\ by Lemma \ref{lem:hoelder-expectation} and \cite[Lemma 4.5]{MR4399156}, so that in the last display $f\circ \tau_{\varphi}$ has a.s.\ compact support. Thus, using Theorem \ref{thm:ext-process-massive-GFF} we have a.s.\ 
\begin{equation}
\lim_{\epsilon \to 0} \E[  e^{-\avg{\extprocGFFms, f\circ \tau_{\varphi}}} ] \bigm|_{\varphi=X_s^{c,0}} = \E[e^{-\avg{\eta^{\GFF,s},f\circ \tau_{\varphi}}}] \bigm|_{\varphi= X_s^{c,0}}.
\end{equation}
Therefore, we get by the dominated convergence theorem and \eqref{eq:calculation-laplace-functionals-phis}
\begin{equation}
\lim_{\epsilon \to 0} \E[e^{-\avg{\extprocPhis, f}} ] 
= \E \Big [  \E[e^{-\avg{\eta^{\GFF,s},f\circ \tau_{\varphi}}} ]\bigm|_{\varphi= X_s^{c,0}} \Big ].
\end{equation}

On the other hand, using the independence between $\eta^{\GFF,s}$ and  $X_s^{c,0}$,
Theorem \ref{thm:ext-process-massive-GFF} and a change of variables in the $h$-coordinate, we have
\begin{align}
\E\Big[ \E[ e^{-\avg{\eta^{\GFF,s},f\circ \tau_{\varphi}}} ] \bigm|_{\varphi=X_s^{c,0}} \Big] 
&=\E\Big[ \E \Big [ \exp\Big(-\int(1-e^{-f(x,h)}) e^{\alpha \varphi(x)} \ZDM_s^{\GFF}(dx) e^{-\alpha h}dh\Big) \Big]_{\varphi= X_s^{c,0} }\Big]=\nnb
&=\E\left[\exp\left(-\int(1-e^{-f(x,h)}) e^{\alpha X_s^{c,0}(x)} \ZDM_s^{\GFF}(dx) e^{-\alpha h}dh\right)\right],
\end{align}
which proves \eqref{eq:laplace-functional-conv-for-eta-s} and that the limit $\eta^s$ is a Poisson point process with the stated intensity measure. 
It follows that $\extprocPhis$ converges as $\epsilon \to 0$, since the Laplace functionals $\cL_s(f)$ converge for every suitable $f\colon \Omega\times \R \to [0,\infty)$.
\end{proof}

\subsection{Continuous field at local extrema: proof of Lemma \ref{lem:r-local-maxima} and Lemma \ref{lem:r-limit-laplace-functional}}
\label{sec:cont-field-at-local-maxima}
In this section we follow closely the exposition in \cite[Section 4]{MR3509015}.
Examining the proofs of Lemma 4.4 and Proposition 4.5 therein we see that  the main ingredients are the fact that the field $\hat h'$ is log-correlated and that $\hat h''$ is continuous and has variance of order unity.  
For a fixed $s>0$ these assumption clearly hold in our case when the roles of $\hat h'$ and $\hat h''$ are played by $\XGFFms$ and $\Xsc$.
Thus, the main purpose of this section is to show that these results continue to hold with our non-Gaussian field $\Xsc$. 

\begin{proof}[Proof of Lemma \ref{lem:r-local-maxima}]
We first prove \eqref{eq:r-local-maxima-1}. For $\lambda >0$ let $\Gamma_\epsilon^s(\lambda)$ be the random set that consists of all points where the field $\Phiaux$ is at least $m_\epsilon - \lambda$ high
and $\Theta_{r\epsilon}^s$ and $\Theta_{r_\epsilon}^s$ the sets of $r\epsilon$-local extrema and $r_\epsilon$-local extrema of $\Phiaux$.
We first prove that 
\begin{equation}
\label{eq:prob-set-not-empty}
\lim_{r\to 0} \lim_{\epsilon \to 0} \P(\Gamma_\epsilon^s(\lambda) \cap (\Theta_{r\epsilon}^s\Delta \Theta_{ r_\epsilon}^s) \neq \emptyset)=0.
\end{equation}
To this end, we first observe that by the properties of the sequence $(r_\epsilon)_\epsilon$ we have $r_\epsilon > r\epsilon$ eventually,
and hence, $\Theta_{r_\epsilon}^s \subseteq \Theta_{r\epsilon}^s$.
Thus, $\Theta_{r\epsilon}^s \Delta \Theta_{r_\epsilon}^s = \Theta_{r\epsilon}^s \setminus \Theta_{r_\epsilon}^s$ for $\epsilon>0$ small enough,
which means that the random set in the probability in \eqref{eq:prob-set-not-empty} contains all points that are $r\epsilon$-local maxima,
but not $r_\epsilon$-local maxima, and where the field is at least $m_\epsilon-\lambda$ high.

If that set is non-empty, then there is a random point $x$ that is an $r\epsilon$-local maximum with $\Phiaux(x) \geq m_\epsilon -\lambda$.
Since $x$ is not an $r_\epsilon$-local maximum, there must be a point $y$, for which we have $\Phiaux(y) \geq \Phiaux(x) \geq m_\epsilon-\lambda$ and 
\begin{equation}
r \epsilon \leq |x-y| \leq r_\epsilon \leq 1/r,
\end{equation}
where the last inequality holds for $\epsilon$ small enough.
Now recall that $\Phiaux \stackrel{d}{=} (\Phi_0^\GFF-\Phi_s^\GFF) + \Phi_s^\SG= \Phi_0^\GFF+\Phi_s^\Delta$, where in the last sum the fields are not independent.
Since $\max_{\Omega_\epsilon} |\Phi_s^\Delta| \leq \tilde C< \infty$ uniformly in $\epsilon$,
the non-Gaussian field $\Phiaux$ is equal to a log-correlated Gaussian field up to a constant.
This means that Theorem \ref{thm:near-maxima} (or rather its generalisation \cite[Lemma 3.3]{MR3729618}) holds also for $\Phiaux$ and thus, \eqref{eq:prob-set-not-empty} follows.

To conclude the proof of \eqref{eq:r-local-maxima-1} we take $\lambda$ large enough such that $f$ vanishes on $\Omega \times (-\infty,-\lambda)$. Then, for $\epsilon$ small enough, the event
\begin{equation}
\{\avg{\extprocPhis,f} - \avg{\eta_{r\epsilon, \epsilon}^s,f} \neq 0\}
\end{equation}
is contained in the event \eqref{eq:prob-set-not-empty}.

The proof of \eqref{eq:r-local-maxima-2} is identical when the sets $\Theta_{r_\epsilon}^s$ and $\Theta_{r\epsilon}^s$ are replaced by $\Theta_{r_\epsilon}^\GFF$ and $\Theta_{r_\epsilon}^\GFF$.
\end{proof}

Throughout the rest of this subsection we gradually work towards the proof of Proposition \ref{prop:changing-sets-of-extrema}
which readily implies Lemma \ref{lem:r-limit-laplace-functional}.

\begin{proposition}[Analogue to {\cite[Proposition 4.5]{MR3509015}}]
\label{prop:changing-sets-of-extrema}
Let $s>0$ and $f\colon \Omega \times \R\to [0,\infty)$ be continuous and with compact support. Then, for any $\delta >0$,
\begin{equation}
\label{eq:changing-set-of-maxima}
\lim_{r\to \infty} \limsup_{\epsilon\to 0} \P\Big( \big| \sum_{ \Theta_{r\epsilon}^s} f(x, \Phiaux(x)-m_\epsilon) - \sum_{\Theta_{r\epsilon}^\GFF} f(x, \Phiaux(x) - m_\epsilon) \big| >\delta\Big)=0.
\end{equation}
\end{proposition} 

To this end, we proceed along similar steps as in the proof of \cite[Proposition 4.5]{MR3509015} with the roles of $\hat h'$ and $\hat h''$ being played by $\XGFFms$ and $\Xsc$.
We first show that the level sets of $\Phiaux$ can be compared to the ones of $\XGFFms$.

\begin{lemma}[Analogue to {\cite[Lemma 4.6]{MR3509015}}]
\label{lem:inclusions-level-sets}
The following limits hold for the level sets $\Gamma_\epsilon^s(\lambda)$ and $\Gamma_\epsilon^\GFF(\lambda)$:
\begin{align}
\label{eq:prob-inclusion-1}
\lim_{\lambda \to \infty} \liminf_{\epsilon \to 0} \, &\P(\Gamma_\epsilon^\GFF(\lambda) \subseteq \Gamma_\epsilon^s(2\lambda) ) =1, \\
\label{eq:prob-inclusion-2}
\lim_{\lambda \to \infty} \liminf_{\epsilon \to 0} \, &\P(\Gamma_\epsilon^s(\lambda) \subseteq \Gamma_\epsilon^\GFF(2\lambda) ) =1.
\end{align}
\end{lemma}

\begin{proof}
We show that in both cases the probability of the complement of the event of interest converges to $0$ when taking the said limits.
For the limit \eqref{eq:prob-inclusion-1} we first observe that 
\begin{align}
\Gamma_\epsilon^\GFF(\lambda) \setminus \Gamma_\epsilon^s(2\lambda) &= \{x\in \Omega_\epsilon \colon \XGFFms (x) \geq m_\epsilon - \lambda, \, \Phiaux(x) < m_\epsilon - 2 \lambda\}  \\
&\subseteq \{x\in \Omega_\epsilon \colon \XGFFms(x) \geq m_\epsilon-\lambda, \, \Xsc(x) < -\lambda \}.
\end{align}
Decomposing the latter event by the size of the set $\Gamma_\epsilon^\GFF(\lambda)$, we obtain
\begin{align}
\P&\big(\Gamma_\epsilon^\GFF(\lambda) \setminus \Gamma_\epsilon^s(2\lambda) \neq \emptyset\big) \leq 
\P\big( \exists x\in \Gamma_\epsilon^\GFF(\lambda) \colon \Xsc(x) < - \lambda\big) \nnb
&\leq  \P\big(|\Gamma_\epsilon^\GFF| > e^{C\lambda}\big) + \P\big( \exists x\in \Gamma_\epsilon^\GFF(\lambda) \colon \Xsc(x) <-\lambda, \, |\Gamma_\epsilon^\GFF(\lambda)| \leq e^{C\lambda}\big).
\end{align}
The first probability converges to $0$ by the lower bound in Theorem \ref{thm:size-level-sets}, see also Proposition \ref{prop:exponential-lower-bound-level-sets}, when applied to $\XGFFms$.
Using independence of $\XGFFms$ and $\Xsc$ together with a union bound, we have for the second term
\begin{align}
&\P\big( \exists x\in \Gamma_\epsilon^\GFF(\lambda) \colon \Xsc(x) <-\lambda, \, |\Gamma_\epsilon^\GFF(\lambda)| \leq e^{c\lambda}\big)\nnb
& =  \E\Big[\P\big( x\in A \colon \Xsc(x) < -\lambda, \, |A|\leq e^{c\lambda} \big)|_{A=\Gamma_\epsilon^\GFF(\lambda)}\Big] \nnb
&\leq\E\Big[ \Big(\sum_{x\in A} \P\big(\Xsc(x) <-\lambda \big) \mathbf{1}_{|A|\leq e^{C\lambda}}\Big)|_{A=\Gamma_\epsilon^\GFF(\lambda)} \Big]\nnb
&\leq e^{c\lambda}  \max_{x\in \Omega_\epsilon}\P(\Xsc(x) < -\lambda).
\end{align}
For the term in last display to converge to $0$, it suffices to prove that $\Xsc = X_s^h + \Phi_s^\SG$ has Gaussian tails uniformly in $x\in \Omega_\epsilon$ and $\epsilon>0$.
This is true for the field $X_s^h$, since it is a Gaussian field and the variances are bounded uniformly in $\epsilon>0$.
Indeed, using \eqref{eq:bounds-fourier-coefficients-Xsh} and the fact that $g_s(\lambda)=O(\frac{1}{1+\lambda^2})$, we have
\begin{equation}
\E | X_s^h(x) |^2 = \sum_{k\in \Omega_\epsilon^*} |\hat q_\epsilon(k)|^2
\leq \sum_{k\in \Omega_\epsilon^*} g_s(\kappa |k|^2+m^2) = \sum_{k\in \Omega^*}O\big(\frac{1}{|k|^4 + 1}\big)< \infty,
\end{equation}
which is  uniform in $x\in \Omega_\epsilon$ and $\epsilon>0$.
The field $\Phi_s^\SG$ is non-Gaussian, but is coupled to the Gaussian field $\Phi_s^\GFF$ up to a constant by Theorem \ref{thm:coupling} and Theorem \ref{thm:coupling-bd}. 
Hence, it suffices to prove that $\Phi_s^\GFF$ has Gaussian tails uniformly in $x\in \Omega_\epsilon$ and $\epsilon>0$. Note that
\begin{equation}
\cov(\Phi_s^\GFF) = \int_s^\infty \dot c_u^\epsilon du = (-\Delta^\epsilon + m^2)^{-1} e^{-s(-\Delta^\epsilon + m^2)}. 
\end{equation}
Hence, we have for any $x\in \Omega_\epsilon$
\begin{equation}
\E|\Phi_s^\GFF(x)|^2 
= \sum_{k\in \Omega_\epsilon^*} \frac{1}{-\hat \Delta^\epsilon(k) + m^2} e^{-s(-\hat \Delta^\epsilon(k) + m^2)} 
\leq \sum_{k\in \Omega^*} \frac{1}{\kappa|k|^2 + m^2} e^{-s(\kappa|k|^2 + m^2)},
\end{equation}
where the bound on the right hand side is finite for $s>0$ uniformly in $x\in \Omega_\epsilon$ and $\epsilon>0$.
Hence, there is a constant $C>0$, such that 
\begin{equation}
\max_{\Omega_\epsilon} \P(\Xsc(x) <-\lambda) \leq e^{-\lambda^2/2C},
\end{equation}
from which the claimed convergence follows when taking $\epsilon \to 0$ and then $\lambda \to \infty$.

To prove \eqref{eq:prob-inclusion-2} we first observe that 
\begin{align}
\Gamma_\epsilon^s(\lambda) &\setminus \Gamma_\epsilon^\GFF(2\lambda) = \{x\in \Omega_\epsilon \colon \Phiaux(x) \geq m_\epsilon - \lambda, \, \XGFFms < m_\epsilon - 2\lambda\} \nnb
& =\bigcup_{n\geq 2} \big\{ x\in \Omega_\epsilon \colon \Phiaux(x) \geq m_\epsilon - \lambda, \, m_\epsilon - (n+1)\lambda \leq \XGFFms < m_\epsilon - n\lambda \big\} \nnb
\label{eq:union-random-sets}
&\subseteq \bigcup_{n\geq 2} \big\{x\in \Omega_\epsilon \colon \XGFFms(x) \geq m_\epsilon - (n+1)\lambda, \, \Xsc(x) > (n-1)\lambda \big\}.
\end{align}
Denote by $A_n$ the event that the random set indexed by $n$ in the union \eqref{eq:union-random-sets} is empty, i.e.\ 
\begin{equation}
A_n= \Big \{ \big\{x\in \Omega_\epsilon \colon \XGFFms(x) \geq m_\epsilon - (n+1)\lambda, \, \Xsc(x) > (n-1)\lambda \big\} = \emptyset \Big \}.
\end{equation}
We also set $B_n=\{ |\Gamma_\epsilon^\GFF(n\lambda)| > e^{\kappa n \lambda} \}$ for some $\kappa >0$. Then 
by \cite[Proposition 4.1]{MR3509015}
 there is $\beta>0$, such that
\begin{equation}
\P(B_n) \leq e^{-\beta\kappa n \lambda}.
\end{equation}
Using a union bound and a decomposition of the event $A_n$ by $B_{n+1}$, we obtain from \eqref{eq:union-random-sets}
\begin{align}
\label{eq:union-bound-A-n}
\P\big(\Gamma_\epsilon^s(\lambda)\setminus \Gamma_\epsilon^\GFF(2\lambda) \neq \emptyset \big)\leq \sum_{n\geq 2} \P(A_n)\leq \sum_{n\geq 2} \big( \P(B_{n+1}) + \P(A_n \cap B_{n+1}^c) \big).
\end{align}
To see that also the second term in the sum in \eqref{eq:union-bound-A-n} is exponentially small we write
\begin{align}
A_n\setminus B_{n+1} = \big\{ \exists x \in \Gamma_\epsilon^\GFF\big((n+1)\lambda \big) \colon \Xsc(x) > (n-1)\lambda, \, |\Gamma_\epsilon^\GFF\big((n+1)\lambda\big) |\leq e^{\kappa (n+1) \lambda} \big \}
\end{align}
and, using again a union bound and conditioning on $\Gamma_\epsilon^\GFF\big((n+1)\lambda\big)$, bound its probability by
\begin{align}
\P(A_n\setminus B_{n+1}) 
&=\E \Big[\P\big(\exists x\in C \colon \Xsc(x) >(n-1)\lambda, \, |C|\leq e^{\kappa(n+1)\lambda} \big) \bigm|_{C=\Gamma_\epsilon^\GFF\big((n+1)\lambda\big)} \Big]=\nnb
& \leq \E\Big[ \big( \sum_{x\in C} \P\big( \Xsc(x)> (n-1)\lambda\big) \mathbf{1}_{|C|\leq e^{\kappa(n+1)\lambda}} \big)\bigm|_{C=\Gamma_\epsilon^\GFF\big( (n+1)\lambda \big)} \Big] \nnb
&\leq  e^{\kappa (n+1)\lambda} \max_{x\in \Omega_\epsilon} \P\big(\Xsc(x)>(n+1)\lambda\big) \leq e^{\kappa (n+1)\lambda} e^{-\big((n-1)\lambda \big)^2/2C}.
\end{align}
Hence, the sum in \eqref{eq:union-bound-A-n} converges uniformly in $n$, and thus, it converges to $0$ when taking the limits $\epsilon \to 0$ and then $\lambda \to \infty$.
\end{proof}

The next result states that the continuous field $\Xsc$ does not fluctuate too much at points where the field $\XGFFms$ is large.
Recall that the oscillation of a function $g$ on a set $A$ is defined by
\begin{equation}
\osc_A g = \max_{x\in A}g(x) - \min_{x\in A} g(x).
\end{equation}

\begin{lemma}[Analogue to {\cite[Lemma 4.7]{MR3509015}}]
\label{lem:oscillation-cont-field}
For any $\lambda>0$, any $\kappa >0$ and any $r\geq 1$,
\begin{equation}
\label{eq:limit-oscillation-cont-field}
\limsup_{\epsilon \to 0} \P\big (\max_{x\in \Gamma_\epsilon^\GFF(\lambda)} \osc_{\Lambda_{2r\epsilon}(x)} \Xsc > \kappa \big)=0
\end{equation}
\end{lemma}

\begin{proof}
Denote $F(x) = \{ \max_{z \in \Lambda_{2r\epsilon}(x)} |\Xsc(z) - \Xsc(x)| > \kappa/2\}$ and note that the event in \eqref{eq:limit-oscillation-cont-field}  is contained in $\bigcup_{x\in \Gamma_\epsilon^\GFF(\lambda)} F(x)$.
For $\lambda >0$ and $M>0$ we define $A_{M}=\{ |\Gamma_\epsilon^\GFF(\lambda)| \leq M\}$. 
Then, by \eqref{eq:level-set-size-large}, we have for its complement 
\begin{equation}
\lim_{M\to \infty} \limsup_{\epsilon \to 0} \P(A_{M}^{c} ) = 0.
\end{equation}
On the other hand, conditioning on $\Gamma_\epsilon^\GFF(\lambda)$, a union bound, Lemma \ref{lem:hoelder-expectation}, \cite[Lemma 4.5]{MR4399156} and Markov's inequality yield
\begin{align}
\P\Big( &\bigcup_{x\in \Gamma_\epsilon^\GFF(\lambda)} F(x) \cap A_{M} \Big) 
=  \E \Big[ \Big( \P\big( \bigcup_{x\in A} F(x) \mathbf{1}_{|A|\leq M}     \big)  \Big) \bigm|_{A=\Gamma_\epsilon^\GFF(\lambda)}  \Big] \nnb
&\leq \E \Big[ \Big(\sum_{x\in A}\P\big( F(x) \mathbf{1}_{|A|\leq M}     \Big) \bigm|_{A=\Gamma_\epsilon^\GFF(\lambda)}  \Big] 
\leq M \max_{x\in \Omega_\epsilon} \P\big(F(x)\big ) \nnb
&\leq M \max_{x\in \Omega_\epsilon} \sum_{z\in \Lambda_{r\epsilon}(x)} \P(|\Xsc(x) - \Xsc(z)| >\kappa/2 )
\leq \frac{2M}{\kappa} \max_{x \in \Omega_\epsilon} \sum_{z\in \Lambda_{r\epsilon}(x) } \E[|\Xsc(x) - \Xsc(z)|] \nnb
&\leq \frac{2M}{\kappa} \max_{\Omega_\epsilon} |\Lambda_{r\epsilon}(x)| C |r\epsilon|^\gamma 
\leq \frac{2M}{\kappa}  r^2 C |r\epsilon|^\gamma,
\end{align}
which converges to $0$ when sending $\epsilon \to 0$.
\end{proof}

To state the next result we introduce the following notation to address the locations of maxima of $\Phiaux$ and $\XGFFms$ around a point $x\in \Omega_\epsilon$:
\begin{equation}
 \Pi^s(x) \coloneqq \argmax_{\Lambda_{2r\epsilon}(x)} \Phiaux \qquad \text{and} \qquad \Pi^\GFF(x) \coloneqq \argmax_{\Lambda_{2r\epsilon}(x)} \XGFFms.
\end{equation}
Note that since the maxima of the fields $\Phiaux$ and $\XGFFms$ are a.s.\ unique, these are a.s.\ well-defined maps $\Omega_\epsilon \to \Omega_\epsilon$.
Our goal is to prove that when restricting to appropriate level sets with high probability every $\Phiaux$ local maximum can be mapped to a $\XGFFms$ local maximum nearby and vice versa.
At this point we use the fundamental Theorem \ref{thm:near-maxima} which is now accessible thanks to the replacement of $r_\epsilon$ by $r\epsilon$ and the additional limit $r\to \infty$.

\begin{lemma}[Analogue to {\cite[Lemma 4.8]{MR3509015}}]
\label{lem:correspondence-maxima}
The following holds with probability tending to one in the limits $\epsilon \to 0$, $\kappa \to 0$ and then $r\to \infty$:
\begin{align}
\label{eq:phi-s-gff-local-max-1}
x\in \Theta_{r\epsilon}^s\cap \Gamma_\epsilon^s(\lambda)\cap \Gamma_\epsilon^\GFF(\lambda) \qquad &\implies \qquad 
\begin{cases}
\Pi^\GFF(x) \in \Theta_{r\epsilon}^\GFF, \, |\Pi^\GFF(x)-x| \leq r\epsilon/2,\\
\text{and~} 0\leq \Phiaux(x) - \Phiaux(\Pi^\GFF(x)) \leq \kappa
\end{cases}
\\
\label{eq:phi-s-gff-local-max-2}
x\in \Theta_{r\epsilon}^\GFF \cap \Gamma_\epsilon^s(\lambda) \cap \Gamma_\epsilon^\GFF(\lambda) \qquad &\implies \qquad 
\begin{cases}
\Pi^s(x) \in \Theta_{r\epsilon}^s, \, |\Pi^s(x)-x| \leq r\epsilon/2,\\
\text{and~} 0\leq \XGFFms(x) - \XGFFms(\Pi^s(x)) \leq \kappa
\end{cases}
\end{align}
\end{lemma}

\begin{proof}
For $\lambda >0$ and $r>0$ denote
\begin{equation}
\label{eq:event-near-maxima-r-2}
A_{\epsilon, \lambda, r}^s \coloneqq \big\{ \forall u,v \in \Gamma_\epsilon^s(\lambda) \colon |u-v| \leq r\epsilon/2 \, \text{~or~} \, |u-v| \geq 2/r \big \}
\end{equation}
and similarly $A_{\epsilon, \lambda, r}^\GFF$ with $\XGFFms$ in place of $\Phiaux$.
Then by Theorem \ref{thm:near-maxima} the probabilities of the complements of these events tend to $0$ when taking the said limits,
so we may assume that the event $A_{\epsilon, \lambda, r}^s\cap A_{\epsilon, \lambda, r}^\GFF$ occurs.

Let $x\in \Theta_{r\epsilon}^s \cap \Gamma_\epsilon^s(\lambda)\cap \Gamma_\epsilon^\GFF(\lambda)$.
Then $\Phiaux(x), \XGFFms(x) \geq m_\epsilon -\lambda$ and hence, since $A_{\epsilon, \lambda, r}^s\cap A_{\epsilon, \lambda, r}^\GFF$ occurs, we have 
\begin{equation}
\label{eq:intermediate-scale}
r\epsilon/2 < |z-x| < 2/r \, \implies \, \Phiaux(z), \XGFFms(z) < m_\epsilon - \lambda.
\end{equation}
Since $\Pi^\GFF(x)$ is the maximum of $\XGFFms$ on $\Lambda_{2r\epsilon}(x)$,
we have $\XGFFms(\Pi^\GFF(x)) \geq m_\epsilon -\lambda$ and hence, $\Pi^\GFF(x) \in \Lambda_{r\epsilon/2}(x)$ by \eqref{eq:intermediate-scale}.
Hence, $\Pi^\GFF(x)=\max_{\Lambda_{r\epsilon}(x)}\XGFFms$, so $\Pi^\GFF(x) \in \Theta_{r\epsilon}^\GFF$ as claimed.
Moreover, we note that by the maximality of $\Pi^\GFF(x)$ and the definition of the oscillation, we have
\begin{equation}
0 \leq \Phiaux(x) - \Phiaux(\Pi^\GFF(x)) \leq \Xsc(x) - \Xsc(\Pi^\GFF(x)) \leq \osc_{\Lambda_{2r\epsilon(x)}} \Xsc.
\end{equation}
By Lemma \ref{lem:oscillation-cont-field} we have with probability one that $\osc_{\Lambda_{2r\epsilon}(x) } \Xsc < \kappa$ when $\epsilon \to 0$. The proof of \eqref{eq:phi-s-gff-local-max-2} is identical when exchanging the roles of $\Phiaux$ and $\XGFFms$.
\end{proof}

Now we have collected all results to give the proof of Proposition \ref{prop:changing-sets-of-extrema}.
\begin{proof}[Proof of Proposition \ref{prop:changing-sets-of-extrema}]
Fix $\delta>0$ and let $f$ be as in the statement.
Since $f$ has compact support, we can choose $\lambda_0>0$, so that $f$ vanishes outside of $\Omega \times [-\lambda_0/2,\infty)$.
Since $f$ is uniformly continuous, we can find for each $M>0$ and $r>0$ some $\kappa>0$ such that for all $\epsilon>0$ small enough and all $x,x' \in \Omega_\epsilon$ and all $h,h'\in \R$,
\begin{equation}
\label{eq:uniform-continuity-f}
|x-x'|\leq \frac{r\epsilon}{2} \text{~and~} |h-h'| \leq \kappa \qquad \implies \qquad
|f(x,h)-f(x',h')| \leq \delta/M.
\end{equation}
Now, for $\kappa>0$ as in \eqref{eq:uniform-continuity-f}, let $A_{\epsilon, M, r, \lambda,\kappa}^s$ be the intersection of the event in \eqref{eq:event-near-maxima-r-2} with the complement of the event in \eqref{eq:limit-oscillation-cont-field}
and $\{ \Gamma_\epsilon^s(\lambda/2)\subseteq \Gamma_\epsilon^\GFF(\lambda)\}$ and $\{|\Gamma_\epsilon^\GFF(\lambda)| \leq M\}$.
Note that on $A_{\epsilon, M, r, \lambda,\kappa}^s$ the implication \eqref{eq:phi-s-gff-local-max-1} holds as shown in Lemma \ref{lem:correspondence-maxima}.
Thus, for $\lambda \geq \lambda_0$ we have on this event
\begin{align}
&\sum_{\Theta_{r\epsilon}^s} f(x, \Phiaux(x)-m_\epsilon) = \sum_{\Theta_{r\epsilon}^s\cap \Gamma_\epsilon^s(\lambda)\cap \Gamma_\epsilon^\GFF(\lambda)} f(x, \Phiaux(x)-m_\epsilon) \nnb
&\leq |\Gamma_\epsilon^\GFF(\lambda)| \frac{\delta}{M} + \sum_{\Theta_{r\epsilon}^s\cap \Gamma_\epsilon^s(\lambda)\cap \Gamma_\epsilon^\GFF(\lambda)} f(\Pi^\GFF(x), \Phiaux(\Pi^\GFF(x))-m_\epsilon) \nnb
& \leq \delta + \sum_{\Theta_{r\epsilon}^\GFF \cap \Gamma_\epsilon^s(\lambda)\cap \Gamma_\epsilon^\GFF(\lambda)} f(x, \Phiaux(x)-m_\epsilon) \nnb
& \leq \delta + \sum_{\Theta_{r\epsilon}^\GFF} f(x,\Phiaux(x)-m_\epsilon).
\end{align}
On the equality we introduced the intersection with $\Gamma_\epsilon^s(\lambda)$ and $\Gamma_\epsilon^\GFF(\lambda)$ under the summation,
since $f$ has compact support and since the event $\{\Gamma_\epsilon^s(\lambda/2) \subseteq \Gamma_\epsilon^\GFF(\lambda)\}$ occurs.
The inequality in the second line follows from \eqref{eq:uniform-continuity-f} together with \eqref{eq:phi-s-gff-local-max-1}.
Finally, we use the estimate $|\Gamma_\epsilon^\GFF(\lambda)|\leq M$ and observe that on the event $A_{\epsilon, M, r, \lambda,\kappa}^s$ we have $\Pi^\GFF(x)\in \Theta_{r\epsilon}^\GFF$ by \eqref{eq:phi-s-gff-local-max-1}.
Hence, we may replace $\Theta_{r\epsilon}^s$ by $\Theta_{r\epsilon}^\GFF$ under the summation.
Note that we also rely on the positivity of $f$, since not every $\XGFFms$ local maximum might arise from the map $\Pi^\GFF$.
The last inequality is immediate using again positivity of $f$.

The reverse inequality follows from exchanging the roles of $\Phiaux$ and $\XGFFms$.
Using identical arguments we have
\begin{equation}
\sum_{\Theta_{r\epsilon}^\GFF} f(x, \Phiaux(x)-m_\epsilon) \leq \delta + \sum_{\Theta_{r\epsilon}^s} f(x, \Phiaux(x)-m_\epsilon) \qquad \text{on~} A_{\epsilon, M, r, \lambda, \kappa}^\GFF,
\end{equation}
where the event $A_{\epsilon, M, r, \lambda,\kappa}^\GFF$ is defined as the intersection of $A_{\epsilon, \lambda,r}^\GFF$,
the complement of the event in \eqref{eq:limit-oscillation-cont-field}, $\{\Gamma_\epsilon^\GFF(\lambda/2) \subseteq \Gamma_\epsilon^s(\lambda) \}$ and $\{\Gamma_\epsilon^s(\lambda)\leq M\}$.

To finish the proof we note that the event in \eqref{eq:changing-set-of-maxima} is contained in $A_{\epsilon, M, r, \lambda,\kappa}^s \cap A_{\epsilon, M, r, \lambda,\kappa}^\GFF$.
Hence, the limit \eqref{eq:changing-set-of-maxima} follows,
since Theorem \ref{thm:near-maxima}, Lemma \ref{lem:oscillation-cont-field}, Theorem \ref{thm:size-level-sets}  and Lemma \ref{lem:inclusions-level-sets} imply
\begin{equation}
\lim_{\lambda\to \infty} \limsup_{r\to \infty} \limsup_{M\to \infty} \limsup_{\kappa \to 0} \limsup_{\epsilon \to 0} \P((A_{\epsilon, M, r, \lambda,\kappa}^s)^c)=0
\end{equation}
and analogously for $A_{\epsilon, M, r, \lambda,\kappa}^\GFF$.
\end{proof}

\begin{proof}[Proof of Lemma \ref{lem:r-limit-laplace-functional}]
Let $A_{\delta, r, \epsilon}$ be the complement of the event in \eqref{eq:changing-set-of-maxima} and denote
\begin{equation}
\Sigma_1 =  \sum_{ \Theta_{r\epsilon}^s} f(x, \Phiaux(x)-m_\epsilon), \qquad
\Sigma_2 = \sum_{\Theta_{r\epsilon}^\GFF} f(x, \Phiaux(x) - m_\epsilon).
\end{equation}
We need to estimate the expectation of the difference $e^{-\Sigma_1}- e^{-\Sigma_2} = e^{-\Sigma_1}(1-e^{-(\Sigma_2-\Sigma_1)})$.
Using that $|\Sigma_1 - \Sigma_2|<\delta$ on $A_{\delta,r, \epsilon}$, we have
\begin{equation}
e^{-\Sigma_1}(1-e^{\delta})\mathbf{1}_{A_{\delta, r,\epsilon}} - e^{-\Sigma_2}\mathbf{1}_{A_{\delta, r,\epsilon}^c}
\leq e^{-\Sigma_1}- e^{-\Sigma_2} 
\leq e^{-\Sigma_1}(1-e^{-\delta}) \mathbf{1}_{A_{\delta, r, \epsilon}} + e^{-\Sigma_1}\mathbf{1}_{A_{\delta, r,\epsilon}^c}.
\end{equation}
Since $\Sigma_1, \Sigma_2 \geq 0$, we further have
\begin{equation}
-(e^{\delta}-1)- \mathbf{1}_{A_{\delta, r,\epsilon}^c}
\leq e^{-\Sigma_1}- e^{-\Sigma_2} 
\leq (1-e^{-\delta}) + \mathbf{1}_{A_{\delta, r,\epsilon}^c}.
\end{equation}
Taking expectation and the limits $\epsilon\to 0$ and then $r\to \infty$, we have by Proposition \ref{prop:changing-sets-of-extrema} 
\begin{equation}
\lim_{r\to\infty} \limsup_{\epsilon\to 0} \big | \E[e^{-\Sigma_1}]-\E[e^{-\Sigma_2}] \big | \leq (1- e^{-\delta} ) + (e^{\delta} -1) = e^{\delta} - e^{-\delta}.
\end{equation}
Since $\delta>0$ was arbitrary, \eqref{eq:r-limit-laplace-functional} follows.
\end{proof}

\subsection{Reduction to the extremal process of $\tilde \Phi_s^\SG$: proof of Lemma \ref{lem:reduction-to-phis}}
\label{sec:proof-of-reduction-to-phi-s}
In this section we comment on the proof of the reduction step from $\Phi_0^\SG$ to $\tilde \Phi_s^\SG$ (and hence to $\Phiaux$).
As indicated above this part shares the main ideas with Section \ref{sec:cont-field-at-local-maxima},
in particular the introduction of a limit $r\to \infty$ in order to apply Theorem \ref{thm:near-maxima}.
The main difference is that while the fields $\XGFFms$ and $\Xsc$ in Proposition \ref{prop:changing-sets-of-extrema} are independent,
the fields $\tilde \Phi_s^\SG$ and $R_s$ are not and hence,
the main task here is to show that the uniform convergence of $R_s$ allows to avoid the use of independence.
In what follows we do not present the proofs of the  intermediate results in full detail but rather comment on how to adjust the proofs of Section \ref{sec:cont-field-at-local-maxima}.
In light of the convergence of Laplace functionals, see Proposition \ref{prop:laplace-functional-ppp},
and the assumption \eqref{eq:convergence-ext-proc-Phis-PPP} we need to see that
\begin{equation}
\label{eq:changing-theta-and-Rs}
\lim_{s\to 0} \limsup_{\epsilon \to 0} \big| \E[ e^{-\sum_{\Theta_{r_\epsilon}^s} f(x,\Phi_0^\SG-R_s - m_\epsilon)}] - \E[ e^{-\avg{\extprocSG,f}}]  \big| = 0.
\end{equation}
Hence, the main work below is devoted to the proof of Proposition \ref{prop:changing-sets-Rs} below,
from which \eqref{eq:changing-theta-and-Rs} follows by similar arguments as in the proof of Lemma \ref{lem:r-limit-laplace-functional}.

\begin{proposition}
\label{prop:changing-sets-Rs}
Let $f\colon \Omega \times \R \to [0,\infty)$ be continuous and with compact support.
Then, for any $\delta >0$, we have
\begin{equation}
\label{eq:changin-sets-Rs}
\lim_{s\to 0}\limsup_{\epsilon\to 0} \P\Big(\big| \sum_{\Theta_{r_\epsilon}^s} f(x, \Phi_0^\SG-R_s-m_\epsilon ) - \sum_{\Theta_{r_\epsilon}^\SG} f(x, \Phi_0^\SG-m_\epsilon) \big | >\delta \Big) = 0.
\end{equation}
\end{proposition}

Similarly as in the proof of Proposition \ref{prop:changing-sets-of-extrema} we first introduce a limit $r\to \infty$ thanks to the Lemma \ref{lem:introducing-r-limit-phi-sg} below.

\begin{lemma}[Analogue to Lemma \ref{lem:r-local-maxima}]
\label{lem:introducing-r-limit-phi-sg}
Let $f\colon \Omega \times \R \to [0,\infty)$ be measurable and with compact support.
Then, for any sequence $(r_\epsilon)_\epsilon$ satisfying $r_\epsilon/\epsilon \to \infty$ and $r_\epsilon \to 0$, 
\begin{align}
\lim_{r\to \infty} \limsup_{\epsilon \to 0} \P\big(\avg{\eta_{r\epsilon, \epsilon}^\SG,f} &\neq \avg{\extprocSG,f} \big)=0, \\
\lim_{r\to \infty} \limsup_{\epsilon \to 0} \P\big( \sum_{\Theta_{r_\epsilon}^s} f(x, \Phi_0^\SG - m_\epsilon) &\neq \sum_{\Theta_{r\epsilon}^s} f(x, \Phi_0^\SG - m_\epsilon) \big)=0.
\end{align}
\end{lemma}

\begin{proof}
The same arguments as in the proof of Lemma \ref{lem:r-local-maxima} can be used noting that independence of the fields involved does not enter.
\end{proof}

This result allows us again to focus on the $r\epsilon$-maxima of the fields involved.
The following Lemma is thus important in the proof of Proposition \ref{prop:changing-sets-Rs}.
\begin{lemma}
\label{lem:changin-sets-Rs-with-r-limit}
Let $f\colon \Omega\times \R\to [0,\infty)$ be continuous and with compact support.
Then we have for any $\delta >0$
\begin{equation}
\lim_{r\to \infty} \limsup_{\epsilon \to 0} \P\Big(\sum_{\Theta_{r\epsilon}^\SG} f(x,\Phi_0^\SG-m_\epsilon) - \sum_{\Theta_{r\epsilon}^s} f(x, \Phi_0^\SG - m_\epsilon) | > \delta\Big)= 0.
\end{equation}
\end{lemma}

\begin{proof}
We first observe that Lemma \ref{lem:inclusions-level-sets} and Lemma \ref{lem:oscillation-cont-field} also hold,
when the fields $(\Phiaux, \XGFFms,\Xsc)$ are replaced by $(\Phi_0^\SG,\tilde \Phi_s^\SG, R_s)$.
Indeed, using the coupling from Theorem \ref{thm:coupling}, we have for $C$ large enough and uniformly in $\epsilon$
\begin{equation}
\tilde \Phi_s^\SG - C \leq \Phi_0^\SG \leq \tilde \Phi_s^\SG + C.
\end{equation}
It follows that for $\lambda $ large enough, we have for all $\epsilon>0$
\begin{align}
\P(\Gamma_\epsilon^s(\lambda)  \subseteq \Gamma_\epsilon^\SG(2\lambda) )=1\\
\P(\Gamma_\epsilon^\SG(\lambda)  \subseteq \Gamma_\epsilon^s(2\lambda) )=1,
\end{align}
which means the limit $\lambda\to \infty$ is not even needed here.
Moreover, we have for any $\lambda >0$, any $\kappa >0$ and any $r\geq 1$
\begin{equation}
\limsup_{\epsilon \to 0} \P(\max_{x\in \Gamma_\epsilon^s(\lambda)} \osc_{\Lambda_{2r\epsilon}(x)} R_s >\kappa) = 0
\end{equation}
using the H\"older continuity of $R_s$ and noting that independence is not used in the proof of Lemma \ref{lem:oscillation-cont-field}.
Thus, Lemma \ref{lem:changin-sets-Rs-with-r-limit} follows,
since Lemma \ref{lem:inclusions-level-sets} and Lemma \ref{lem:oscillation-cont-field} are the only results that enter the proof of Proposition \ref{prop:changing-sets-of-extrema}.  
\end{proof}

\begin{proof}[Proof of Proposition \ref{prop:changing-sets-Rs}]
Let $f$ be as in the statement.
Using a triangle inequality we see that the event in \eqref{eq:changin-sets-Rs} is contained in the union of 
\begin{align}
\label{eq:event-1}
\Big\{\big| \sum_{\Theta_{r_\epsilon}^s} f(x, \Phi_0^\SG-R_s-m_\epsilon ) - \sum_{\Theta_{r_\epsilon}^s} f(x, \Phi_0^\SG-m_\epsilon) \big |>\delta/2 \Big\}
\\
\label{eq:event-2}
\Big\{ \big| \sum_{\Theta_{r_\epsilon}^s} f(x, \Phi_0^\SG-m_\epsilon ) - \sum_{\Theta_{r_\epsilon}^\SG} f(x, \Phi_0^\SG-m_\epsilon) \big |>\delta/2 \Big\}.
\end{align}
To see that the probability of \eqref{eq:event-1} converges to $0$ in the said limits, we first observe that since $f$ has compact support there is $\lambda>0$ such that
\begin{equation}
\supp f \subseteq \Omega \times [-\lambda, \infty),
\end{equation}
and thus, the sum in \eqref{eq:event-1} is in fact over $\Theta_{r_\epsilon}^s \cap \Gamma_\epsilon^\SG(\lambda-C)$ with $C> \max_{\Omega_\epsilon}|R_s|$. 
Using that $f$ is uniformly continuous and the fact that $\max_{\Omega_\epsilon} R_s\to 0$ as $s\to 0$ uniformly in $\epsilon>0$, the left side in \eqref{eq:event-1} can be estimated by
\begin{equation}
\label{eq:estimate-event-1}
\sum_{\Theta_{r_\epsilon}^s \cap \Gamma_\epsilon^\SG( \lambda-C)  }\big| f(x,\Phi_0^\SG - R_s -m_\epsilon) - f(x,\Phi_0^\SG-m_\epsilon)   \big| \leq |\Gamma_\epsilon^\SG( \lambda-C) | o(1)
\end{equation}
with $o(1) \to 0$ uniformly in $\epsilon$ as $s\to 0$.
But by \eqref{eq:level-set-size-large} applied to the field $\Phi_0^\SG$ the probability of the event on the right hand side of \eqref{eq:estimate-event-1} converges to $0$, since
\begin{equation}
\P(|\Gamma_\epsilon^\SG(\lambda)| o(1) > \delta/2  ) 
\leq o(1).
\end{equation}

To see that also the probability of the event in \eqref{eq:event-2} converges to $0$, we use Lemma \ref{lem:changin-sets-Rs-with-r-limit} and Lemma \ref{lem:introducing-r-limit-phi-sg}
which allow us to cater the different sets of vertices in the summations.
\end{proof}

Combining the results from this section, we can now give the proof of Lemma \ref{lem:reduction-to-phis}.
\begin{proof}[Proof of Lemma \ref{lem:reduction-to-phis}]
We first prove the existence of the random measure $\ZDM^\SG(dx)$ and the claimed convergence in distribution.
To this end, we prove that for any continuous and non-negative function $g\colon \Omega \to \R$,
the sequence $\big(\ZDM_s(g)\big)_s$ is tight, where we use the notation
\begin{equation}
\ZDM_s(g) \equiv \int_{\Omega} g(x) \ZDM_s(dx).
\end{equation}
By linearity we may assume that $0\leq g <1$ and define for $h_0 \in \R$
\begin{align}
\label{eq:non-cont-f}
f(x,h) = 
\begin{cases}
0  \qquad & h < h_0 - 1 \text{ and } h > 2h_0 + 1
,\\
-\log (1-g)  \qquad &h_0 \leq h  \leq 2 h_0
\end{cases}
\end{align}
with a linear interpolation on $h\in [h_0-1, h_0]$ and $h \in [2h_0, 2h_0+1]$.
Since $f$ as defined in \eqref{eq:non-cont-f} is non-negative, continuous and with compact support, \eqref{eq:convergence-ext-proc-Phis-PPP} holds for $f$.
Note that the expectation on the right hand side of \eqref{eq:convergence-ext-proc-Phis-PPP} can be bounded from above by
\begin{equation}
\E\Big[\exp\big( -\ZDM_s(g)  \alpha^{-1} (e^{-\alpha h_0} - e^{-\alpha 2 h_0} )  \big)\Big],
\end{equation}
while the expectation on the left hand side of \eqref{eq:convergence-ext-proc-Phis-PPP} can be bounded from below by
\begin{align}
\E\Big[\exp\Big({\sum_{x\in \Theta_{r_\epsilon}^s } \log (1-g(x)) \mathbf{1}_{\{2h_0 +1 > \Phi_0^\SG(x) - R_s(x)-m_\epsilon > h_0-1 \}}  }\Big)\Big] &\geq \P(\max_{\Omega_\epsilon}(\Phi_0^\SG - R_s) - m_\epsilon \leq h_0 - 1\ ) \nnb
&\geq \P(\max_{\Omega_\epsilon} \Phi_0^\GFF - m_\epsilon \leq h_0 - 1- C ),
\end{align}
where we used that on the event $\{\max_{\Omega_\epsilon}(\Phi_0^\SG - R_s) - m_\epsilon \leq h_0-1\}$ the random variable in the ex\-pect\-ation is equal to $1$ and  that $\max_{\Omega_\epsilon}|\Phi_s^\Delta| \leq C$ uniformly in $\epsilon>0$ and $s\geq 0$ by Theorem \ref{thm:coupling-bd}. Thus, we have
\begin{equation}
\limsup_{\epsilon\to 0} \P(\max_{\Omega_\epsilon} \Phi_0^\GFF - m_\epsilon \leq h_0 -1 - C ) \leq \E\Big[\exp\big( -\ZDM_s(g)  \alpha^{-1} (e^{-\alpha h_0} - e^{-\alpha 2 h_0} )  \big)\Big].
\end{equation}
From here on, the same arguments involving the tightness of $\big(\max_{\Omega_\epsilon} \Phi_0^\GFF - m_\epsilon\big)_\epsilon$ as in the proof of \cite[Lemma 4.1]{MR4399156} can be used to show tightness of $\big(\ZDM_s(g)\big)_s$.

Moreover, \eqref{eq:changing-theta-and-Rs} implies that the Laplace transforms of $\ZDM_s(g)$ converge as $s\to 0$
and hence, L\'evy's continuity theorem for Laplace transforms implies that $\big(\ZDM_s(g)\big)_s$ converges weakly.
Since $\Omega$ is a second countable locally compact Hausdorff space,
Prohorov's theorem for random measures as stated in \cite[Theorem 4]{MR571672} implies the existence of a random measure $\ZDM^\SG(dx)$ such that as $s\to 0$
\begin{equation}
\label{eq:convergence-Z-s-to-Z-SG}
\ZDM_s(dx) \to \ZDM^\SG(dx)
\end{equation}
in distribution.

We now prove that $\ZDM^\SG(A)>0$ a.s.\ for every non-empty open $A \subseteq \Omega$. To this end, we first observe that by Theorem \ref{thm:convergence-laplace-functionals-phis}
we have for any Borel set $A\subseteq \Omega$ and $x\in \R$
\begin{equation}
\label{eq:max-over-set-A}
\lim_{\epsilon\to 0} \P\big( \max_{A\cap \Omega_\epsilon}\Phiaux - m_\epsilon \leq x \big) = \E[e^{-\ZDM_s(A)\frac{1}{\alpha}e^{-\alpha x}}].
\end{equation}
If $A$ is non-empty and open, then we have by Theorem \ref{thm:convergence-laplace-functionals-phis} that $\ZDM_s(A)>0$ a.s.
Thus, for such $A$, the right hand side in \eqref{eq:max-over-set-A} is a distribution function,
since we have by the dominated convergence theorem as $x\to -\infty$ 
\begin{equation}
\E[e^{-\ZDM_s(A)\frac{1}{\alpha}e^{-\alpha x}}] \to \P\big(\ZDM_s(A)=0\big) = 0.
\end{equation}
In particular, the sequence $\big( \max_{A\cap \Omega_\epsilon} \Phiaux - m_\epsilon\big)_\epsilon$ is tight,
since the right hand side in \eqref{eq:max-over-set-A} is a distribution function.
By Theorem \ref{thm:coupling-bd}, we also have that the sequence $\big( \max_{A\cap \Omega_\epsilon} \Phi_0^\SG - m_\epsilon\big)_\epsilon$ is tight.
Thus, for a given $\kappa >0$, we can choose $x_0$ small enough, such that we have for all $x<x_0$
\begin{equation}
\label{eq:ZDM-positive-tightness-max-over-set-A}
\lim_{\epsilon\to 0} \P\big( \max_{A\cap \Omega_\epsilon}\Phiaux - m_\epsilon \leq x \big) \leq \limsup_{\epsilon\to 0} \P (\max_{A\cap \Omega_\epsilon} \Phi_0^\SG - m_\epsilon \leq x+C) \leq \kappa
\end{equation}
Let now $p=\P(\ZDM^\SG(A)=0)$. By the weak convergence $\ZDM_s(A)\to\ZDM^\SG(A)$ and \eqref{eq:max-over-set-A}, we then have
\begin{align}
\label{eq:ZDM-positive-prob}
p  + \E\big[e^{-\ZDM^\SG(A)\frac{1}{\alpha}e^{-\alpha x}}\mathbf{1}_{\{ \ZDM^\SG(A)>0 \} } \big ] 
 &= \E\big[e^{-\ZDM^\SG(A)\frac{1}{\alpha}e^{-\alpha x}} \big ]  
 = \lim_{s\to 0} \lim_{\epsilon\to 0} \P\big( \max_{A\cap \Omega_\epsilon}\Phiaux - m_\epsilon \leq x \big) \nnb
 &\leq  \limsup_{\epsilon\to 0} \P (\max_{A\cap \Omega_\epsilon} \Phi_0^\SG - m_\epsilon \leq x+C) \leq \kappa. 
\end{align}
By the dominated convergence theorem, the expectation on the left hand of \eqref{eq:ZDM-positive-prob} vanishes when taking $x\to -\infty$, which proves that $p<\kappa$.
Since $\kappa$ was arbitrary, it follows that $p=0$.

Finally, the claimed convergence \eqref{eq:convergence-extproc-SG} follows from taking the limit $s\to 0$ in \eqref{eq:laplace-functional-conv-for-eta-s} together with \eqref{eq:convergence-Z-s-to-Z-SG} and \eqref{eq:changing-theta-and-Rs}.
\end{proof}

\appendix
\section{Extremal process associated Gaussian free fields on the torus}
\label{app:generalisation-to-massive-gff}
In this section we comment on the proof of Theorem \ref{thm:ext-process-massive-GFF}, which is nearly identical to \cite[Theorem 1.1]{MR3509015}, except that the setting is slightly different:
instead of a massless Gaussian field on the unit box $[0,1]^2$ with Dirichlet boundary condition,
the field $\XGFFms$ has mass $\sqrt{m^2 +1/s}$ and periodic boundary condition. Recall that the regularised extremal process $\extprocGFFms$ is given by
\begin{equation}
\label{eq:definition-extremal-process-GFF-ms}
\extprocGFFms = \sum_{x\in \Theta_{r_\epsilon}^\GFF} \delta_{(x, \XGFFms(x) - m_\epsilon)},
\end{equation}
where $\Theta_{r_\epsilon}^\GFF$ is the set of $r_\epsilon$-local maxima defined analogously to \eqref{eq:SG-r-local-max}, with $r_\epsilon$ such that $r_\epsilon\to 0$ and $r_\epsilon/\epsilon \to \infty$. 

In what follows we carefully examine the proof of  \cite[Theorem 1.1]{MR3509015} and verify that all steps also work with the massless Gaussian free field with Dirichlet boundary condition replaced by $\XGFFms$.

For the GFF on the unit box $[0,1]^2$, the extreme values are with high probability attained in the bulk, i.e.\ away from the boundary.
To make this rigorous, a somewhat technical analysis of field values close to the boundary is necessary.
In our case, this discussion is not necessary and hence, the proof is expected to simplify in this regard.

Further, we note that the mass term $\sqrt{m^2 + 1/s}$ changes the covariance of $\XGFFms$ compared to the massless GFF in two minor ways.
Firstly, it adds a constant term depending on $m,s$ to correlations on small scales.
Secondly, the decomposition of $\XGFFms$ into a coarse and a fine field, which is used in \cite{MR3433630}, is slightly different.
We record these properties of the field $\XGFFms$ in the following two lemmas.
\begin{lemma}[Analogue to {\cite[Lemma 2.6]{MR3509015}}]
\label{lem:covariance-XsGFF}
Let $c_{m,s}^\epsilon = \cov(X_s^\GFF)$. Then we have for $x,y \in \Omega_\epsilon$ with $|x-y|\sqrt{m^2+1/s} \leq 1$
\begin{align}
\label{eq:covariance-estimate-XsGFF-1}
c_{m,s}^\epsilon (x,x) &= \frac{1}{2\pi} \log\frac{1}{\epsilon} + O_{m,s}(1), \\
\label{eq:covariance-estimate-XsGFF-2}
c_{m,s}^\epsilon(x,y) &= -\frac{1}{2\pi} \log |x-y| + O_{m,s}(1).
\end{align}
\end{lemma}

\begin{proof}
The first statement \eqref{eq:covariance-estimate-XsGFF-1} is precisely \cite[Lemma 3.3]{MR4303014} with $t=\infty$,
rescaling $x\to x/\epsilon$ and replacing $m$ by $\sqrt{m^2+1/s}f$, i.e.\
\begin{equation}
c_{m,s}^\epsilon (x,x) = \frac{1}{2\pi} \log\frac{1}{\epsilon\sqrt{m^2+1/s}} + O(1)
= \frac{1}{2\pi} \log\frac{1}{\epsilon} + O_{m,s}(1).
\end{equation}

The estimate \eqref{eq:covariance-estimate-XsGFF-2} can be obtained from \cite[Lemma A.2]{MR4303014} with the same adjustments as before, i.e.\
\begin{equation}
c_{m,s}^\epsilon (x,y) = - \frac{1}{2\pi} \log \big( |x-y|\sqrt{m^2+1/s}  \wedge 1 \big) + O(1)
= - \frac{1}{2\pi} \log|x-y| + O_{m,s}(1).
\end{equation}
\end{proof}

To state the second result we first recall the box decomposition from \cite{MR494541}; see also \cite[Section 4.2]{MR4399156} for additional details.
Let $\Gamma \subseteq \Omega$ be the union of horizontal and vertical lines intersecting at the vertices $\frac{1}{K}\Z^2 \cap \Omega$ that subdivides $\Omega$ into boxes $V_i \subset \Omega$, $i=1,\dots, K^2$ of side length $1/K$,
such that $\Omega = \cup_{i=1}^{K^2} V_i \cup \Gamma$.
Assuming that $1/K$ is a multiple of $\epsilon$,
the grid $\Gamma$ can be regarded as a subset of $\Omega_\epsilon$ and
$\Omega_\epsilon = \cup_{i=1}^{K^2} (V_i \cap \Omega_\epsilon) \cup \Gamma$.
Below, we denote by $V_i$ both the subset of $\Omega$ and the corresponding lattice version as subset of $\Omega_\epsilon$. 
Moreover, we define for $\delta \in (0,1)$ and $i\in \{1, \ldots, K^2\}$
\begin{equation}
\label{eq:box-with-boundary}
  V_i^{\delta}= \{x \in V_i \colon \dist(x, \Gamma) \geq \delta/K \}.
\end{equation}

\begin{lemma}
\label{lem:decomposition-gff-boxes}
The field $\XGFFms$ on $\Omega_\epsilon$ admits a decomposition $\XGFFms \stackrel{d}{=} \tilde X_{s,K}^f + \tilde X_{s,K}^c$,
where $\tilde X_{s,K}^f|_{V_i}, i=1,\ldots, K^2$ is a collection of independent Gaussian free fields with mass $\sqrt{m^2+1/s}$ and zero boundary condition on $\partial V_i$ and $\tilde X_{s,K}^c$ is independent with
\begin{equation}
\label{eq:covariance-coarse-GFF}
 \cov(\tilde X_{s,K}^c) =  (-\Delta+m^2 + 1/s)^{-1} - (-\Delta_\Gamma+m^2 + 1/s)^{-1}.
\end{equation}
\end{lemma}

\begin{proof}
Let $\Delta_\Gamma$ be the Laplacian on $\Omega$ with Dirichlet boundary conditions on $\Gamma$,
and let $\Delta$ be the Laplacian with periodic boundary conditions on $\Omega$.
Then, as in \cite[Section 4.2]{MR4399156} we have $-\Delta_\Gamma \geq -\Delta$
and thus, $(-\Delta+m^2 + 1/s)^{-1} \geq (-\Delta_\Gamma+m^2 + 1/s)^{-1}$ as quadratic form inequalities.
Hence, we can choose independent fields $\tilde X_{s,K}^f$ and $\tilde X_{s,K}^c$, with covariances
\begin{align}
  \label{eq:tilde-Xf}
  \cov(\tilde X_{s,K}^f) &= (-\Delta_\Gamma + m^2 + 1/s)^{-1}\\
    \label{eq:tilde-Xc}
  \cov(\tilde X_{s,K}^c) &=  (-\Delta+m^2 + 1/s)^{-1} - (-\Delta_\Gamma+m^2 + 1/s)^{-1},
\end{align}
such that $\XGFFms \stackrel{d}{=} \tilde X_{s,K}^f + \tilde X_{s,K}^c$.
Moreover, since $\Delta_\Gamma$ has Dirichlet boundary condition, the field $\tilde X_{s,K}^f$ has the claimed properties.
\end{proof}

Since the mass term $\sqrt{m^2 + 1/s}$ only leads to an additive constant in Lemma \ref{lem:covariance-XsGFF},
it is clear that it becomes irrelevant for the small scale correlations for a fixed $s>0$ when taking first $\epsilon \to 0$.
Likewise, when sending $K\to \infty$ with $m,s$ fixed,
we may replace the massive Gaussian free fields on the boxes $V_i$ by true massless Gaussian free fields with Dirichlet boundary condition thanks to Lemma \ref{lem:G-concentration} below.
For a proof, we refer to \cite[Lemma 4.2]{MR4399156}.

\begin{lemma}
\label{lem:G-concentration}
Let $G$ be a centred Gaussian field with covariance 
$(-\Delta_\Gamma)^{-1} - (-\Delta_\Gamma + m^2 + 1/s)^{-1}$.
Then for $K^2 \geq m^{2}+ 1/s$,
\begin{equation}
\P\pB{ \max_{\Omega_\epsilon} G > t}
\lesssim K^2 \exp\left(-\frac{ct^2}{\sigma^2}\right),
\qquad
\sigma^2 = \frac{m^2+1/s}{K^2} + e^{-cK^2s}.
\end{equation}
In particular, $\P(\max G > u) \to 0$ for any $u>0$ as $K\to\infty$ with $m$ and $s$ fixed. 
\end{lemma}

The main steps in the proof of \cite[Theorem 1.1]{MR3509015} are \cite[Theorem 3.1]{MR3509015},
which establishes a distributional invariance for the limiting points of the sequence $(\extprocGFFms)_\epsilon$, and \cite[Theorem 3.4]{MR3509015},
which proves that all such limits have the same intensity measures.

In what follows we review the proofs of these results in detail, highlight the parts where the boundary condition and the covariance of the field enters
and explain why these arguments also hold under our assumptions.
We start with the distributional invariance.

\subsection{Distributional invariance of limit points and convergence to a Poisson point process}
\label{app:distributional-invariance}
Let $\eta$ be a subsequential limit of the sequence $(\extprocGFFms)_\epsilon$.
Since  the total mass of such a limit point is infinite almost surely,
we may write $\eta = \sum_{i\in \N} \delta_{(x_i, h_i)}$, where $\delta$ denotes the Dirac measure and $(x_i,h_i)\in \Omega\times \R$ are random.
Analogously to \cite[Section 3.1]{MR3509015}, we need to see that that the law of $\eta$ is invariant under the addition of independent Brownian motions with drift $\frac{\alpha}{2} t$ for $\alpha=\sqrt{8\pi}$ to the height component.
More precisely, letting
\begin{equation}
\eta_t = \sum_{i\in \N} \delta_{(x_i, h_i + B_t^i -\frac{\alpha}{2}t)},
\end{equation}
where $(B_t^i)_{t\geq0}$ are independent standard Brownian motions, our goal is to prove $\eta_t \stackrel{d}{=}\eta$ for $t\geq 0$.

To state the following theorem we define for any measurable function $f\colon \Omega \times \R\to [0,\infty)$
\begin{equation}
\label{eq:f-with-Brownian-motion}
f_t(x,h) = -\log \E [e^{-f(x,h+ B_t- \frac{\alpha}{2} t)}].
\end{equation}

\begin{theorem}[Analogue to {\cite[Theorem 3.1]{MR3509015}} ]
\label{thm:distributional-invariance}
Let $\eta$ be a subsequential distributional limit of the sequence $(\extprocGFFms)_\epsilon$.
Then for any continuous function $f\colon \Omega\times \R \to [0,\infty)$ with compact support and all $t\geq 0$,
\begin{equation}
\label{eq:distributional-invariance}
\E[ e^{-\avg{\eta,f}} ] =\E[ e^{-\avg{\eta, f_t}} ].
\end{equation}
\end{theorem}

\begin{proof}
We verify that the arguments in \cite[Section 4.2]{MR3509015} are also valid under our assumptions.
The main idea is to decompose the field of interest into a sum of independent fields $\hat h'$ and $\hat h''$, so that $\hat h'$ captures the extremal behaviour of the original field and $\hat h''$ is has variance of order $1$.
For the massless GFF with Dirichlet boundary condition,
this decomposition is given in \cite[(4.29)]{MR3509015}.
In our case the analogous choices for $\hat h'$ and $\hat h''$ are
\begin{equation}
\label{eq:decomposition-XGFFms}
\hat h' = \sqrt{1 -\frac{t}{\frac{1}{2\pi} \log \frac{1}{\epsilon}} } h'
\, \text{ and } \, 
\hat h''= \sqrt{\frac{t}{\frac{1}{2\pi} \log \frac{1}{\epsilon}}} h'',
\end{equation}
where $h'$ and $h''$ are independent with $h', h'' \stackrel{d}{=}\XGFFms$.
Note that we have
\begin{equation}
\XGFFms \stackrel{d}{=} \hat h'+ \hat h''.
\end{equation}
We denote the sets of $r_\epsilon$-local extrema with respect to $\XGFFms$ and $h'$ by $\Theta_{r_\epsilon}$ and $\Theta_{r_\epsilon}'$.

The first step is to introduce an additional limit $r\to \infty$ and replace the extremal process $\extprocGFFms$ by $\eta_{r\epsilon, \epsilon}^{\GFF,s}$; see \cite[Lemma 4.4]{MR3509015}.
This approach was also used in various parts of this work, e.g.\ in Lemma \ref{lem:r-local-maxima} and Lemma \ref{lem:introducing-r-limit-phi-sg}.
In this argument the covariance of $\XGFFms$ enters only through the application of \cite[Theorem 2.4]{MR3509015},
which is Theorem \ref{thm:near-maxima} in our case.
Note that the latter result holds for $\XGFFms$ thanks to \cite[Lemma 3.3]{MR3729618},
since the covariance of $\XGFFms$ satisfies the assumption \cite[(A1)]{MR3729618} by Lemma \ref{lem:covariance-XsGFF}.

The next result to verify is \cite[Proposition 4.5]{MR3509015}, which guarantees that in the definition of the extremal process $\extprocGFFms$ in \eqref{eq:definition-extremal-process-GFF-ms},
we may replace the random set $\Theta_{r_\epsilon}$ by $\Theta_{r_\epsilon}'$.
Its proof relies on \cite[Lemma 4.6, Lemma 4.7 and Lemma 4.8]{MR3509015}.

In the proof of \cite[Lemma 4.6]{MR3509015}, the first part,
which establishes \cite[(4.35)]{MR3509015}, can be used without adjustment only noting that the bound in $\var(\hat h''(x)) < C$ now depends now on $m$ and $s$.
For the proof of \cite[(4.36)]{MR3509015} in our case, we set
\begin{equation}
a_\epsilon = \sqrt{1 - \frac{t}{\frac{1}{2\pi}\log \frac{1}{\epsilon}}}
\end{equation}
and define the event $A_n$ analogously to our main reference, and let $B_n = \{ |\Gamma_\epsilon'(n\lambda)| > e^{\tau n\lambda}\}$.
Here, $\Gamma_\epsilon'(n\lambda)$ denote the level sets of $\hat h'$.
To establish the bound
\begin{equation}
\sup_{\epsilon>0} \P(B_n) \leq e^{-c\tau n \lambda},
\end{equation}
we need to generalise \cite[Proposition 4.1]{MR3509015}, which we postpone to Section \ref{app:size-level-sets}; see Proposition \ref{prop:exponential-lower-bound-level-sets}.
From this point on, the rest of the argument can be used without further adjustments.

Next, we argue that \cite[Lemma 4.7]{MR3509015} also holds for $\XGFFms$.
To this end, we define the event $A_{\epsilon, M}'$ as in \cite[(4.45)]{MR3509015}.
The second event in \cite[(4.45)]{MR3509015} is not needed in our case,
since the field $\XGFFms$ has periodic boundary condition.
By Proposition \ref{prop:exponential-lower-bound-level-sets} we then have that
\begin{equation}
\lim_{M\to \infty} \limsup_{\epsilon\to 0} \P\big( (A_{\epsilon, M}')^c\big)=0.
\end{equation}  
To conclude, we use the same estimate as in \cite[(4.47)]{MR3509015} together with
\begin{equation}
\sup_{|x-y|\leq 2r\epsilon} \E[|\hat h_x'' - \hat h_y''  |^2] \lesssim_{m,s,r} \frac{1}{ \log\frac{1}{\epsilon}},
\end{equation}
which holds by Lemma \ref{lem:covariance-XsGFF}.

The proof of \cite[Lemma 4.8]{MR3509015} only uses properties of the field through Theorem \ref{thm:near-maxima},
which, as we argued above, also holds with our assumptions.

Finally, the proof of \cite[Proposition 4.5]{MR3509015} can be used without changes,
since the only additional step where properties of the field enter, is once again the application of  Theorem \ref{thm:near-maxima}.

We now move to the discussion of the proof of \eqref{eq:distributional-invariance}.
As in the proof of \cite[Theorem 3.1]{MR3509015}, we assume
\begin{equation}
\max_{\Omega_\epsilon} |h'|, \max_{\Omega_\epsilon} |h''| \leq 3 \frac{1}{\sqrt{2\pi}} \log \frac{1}{\epsilon}.
\end{equation}
This event occurs with high probability thanks to the tightness of $(\max_{\Omega_\epsilon} \XGFFms - m_\epsilon)_\epsilon$, which holds by \cite[Theorem 1.2]{MR3729618}.
Conditional on $|h'(x) - m_\epsilon| \leq \lambda$ we then have by the Taylor expansion $\sqrt{1-a} = 1 - \frac{1}{2}a + O(a^2)$ 
applied to the prefactor of $h'$ in \eqref{eq:decomposition-XGFFms} with $a = \frac{t}{ g \log \frac{1}{\epsilon}}$, $g=\frac{1}{2\pi}$,
\begin{align}
\XGFFms (x) 
&= \Big[ 1- \frac{1}{2} \frac{t}{g \log\frac{1}{\epsilon}} 
+ O_t \Big( \frac{1}{\big(\log\frac{1}{\epsilon}\big)^{2}} \Big) \Big] h' (x)
+ \sqrt{\frac{t}{g \log \frac{1}{\epsilon}} } h''(x) \nnb
&= h'(x) - \frac{1}{2} t \frac{m_\epsilon}{g \log \frac{1}{\epsilon}} + O_{t,\lambda} \Big(\frac{1}{\log \frac{1}{\epsilon}}\Big)  +  \sqrt{\frac{t}{g \log \frac{1}{\epsilon}}}  h''(x) \nnb
& = h'(x) - \frac{\alpha}{2} t + \sqrt{\frac{t}{g \log \frac{1}{\epsilon}}}  h''(x) +  O_{t,\lambda}\Big( \frac{1}{\sqrt{ \log{\frac{1}{\epsilon}}} }\Big).
\end{align}
On the last line we used $\alpha = \sqrt{8\pi}$ and
\begin{equation}
\frac{m_\epsilon}{g \log \frac{1}{\epsilon}} 
= \frac{2\frac{1}{\sqrt{2\pi}} \log \frac{1}{\epsilon}}{\frac{1}{2\pi}\log\frac{1}{\epsilon}}
+ O\Big(\frac{\log \log \frac{1}{\epsilon}}{\log \frac{1}{\epsilon}} \Big)
= \alpha + O\Big(\frac{1}{\sqrt{\log\frac{1}{\epsilon}}}\Big).
\end{equation}
Thus, there is a constant $C_{t,\lambda}$, which does not depend on $\epsilon$, such that
\begin{equation}
\Bigm| \XGFFms (x) - \Big (h'(x) - \frac{\alpha}{2} t + \sqrt{\frac{t}{\frac{1}{2\pi} \log \frac{1}{\epsilon}}}  h''(x)  \Big) \Bigm| \leq \frac{C_{\lambda, t}}{\log \frac{1}{\epsilon}}.
\end{equation}
Next, we define the random set of vertices $\Delta'(\lambda) = \{ x\in \Omega_\epsilon \colon |h'(x) - m_\epsilon|\leq \lambda \}$.
Compared to the proof of \cite[Theorem 3.1]{MR3509015} the restriction to vertices away from the boundary is not necessary.
Then we have analogously to \cite[(4.61)]{MR3509015} for any $\kappa > 0$
\begin{equation}
\lim_{\lambda\to \infty} \limsup_{r\to \infty} \limsup_{\epsilon \to 0} 
\P\Big(  \sum_{\Theta'_{r_\epsilon} \setminus \Delta'(\lambda)} f(x, \XGFFms(x) - m_\epsilon ) >\kappa \Big) = 0,
\end{equation}
where we also use the tightness of the centred maximum of $h'$.

The argument leading to \cite[(4.62)]{MR3509015} only uses the covariance of the field through Theorem \ref{thm:size-level-sets},
from which we only need the lower bound.
Thanks to Proposition \ref{prop:exponential-lower-bound-level-sets} this also holds for $\XGFFms$.

In the rest of the proof, properties of the field $\XGFFms$ only enter through Proposition \ref{prop:exponential-lower-bound-level-sets} and Theorem \ref{thm:near-maxima},
which applies to $\XGFFms$ by the arguments above.
This concludes the proof of \eqref{eq:distributional-invariance}.
\end{proof}

The distributional invariance of the limit points is then used to prove that every such limit point is a Poisson point process; see \cite[Theorem 3.2]{MR3509015}.
We need to verify that the same is true when the extremal process is constructed from $\XGFFms$.
The corresponding statement is as follows.

\begin{theorem}[Analog to {\cite[Theorem 3.2]{MR3509015}}]
\label{thm:extracting-poisson-law}
Suppose $\eta$ is a point process on $\Omega\times \R$ such that \eqref{eq:distributional-invariance} holds for $f_t$ as in \eqref{eq:f-with-Brownian-motion} and some $t>0$ and all continuous $f\colon\Omega\times \R \to [0,\infty)$ with compact support.
Further, assume that $\eta(\Omega\times [0,\infty)) <\infty$ and $\eta(\Omega\times \R) >0$ a.s.
Then there is a random Borel measure $\ZDM$ on $\Omega$ satisfying $\ZDM(\Omega) \in (0,\infty)$ a.s., such that
\begin{equation}
\label{eq:extracting-poisson-law}
\eta \stackrel{d}{=} \PPP(\ZDM(dx) \otimes e^{-\alpha h}dh).
\end{equation} 
\end{theorem}

\begin{proof}
The proof in \cite[Section 3.2]{MR3509015} can be used almost line by line,
since only \eqref{eq:distributional-invariance}, and no particular property of the field $\XGFFms$ enters.
The argument of the proof is not affected when replacing the square $[0,1]^2$ by $\Omega$.
\end{proof}

\subsection{Uniqueness of the intensity measure}
From Section \ref{app:distributional-invariance} we know that every distributional limit of the sequence $(\extprocGFFms)_\epsilon$ is a Poisson point process on $\Omega \times \R$
with an intensity measure of the form $\ZDM(dx) \otimes e^{-\alpha h} dh$ for some random measure $\ZDM(dx)$ on $\Omega$.
The purpose of this section is to argue that the random measure $\ZDM(dx)$ is the same for all subsequences,
thereby establishing the existence a unique process $\eta^{\GFF,s}$ with $\extprocGFFms \to \eta^{\GFF,s}$ weakly as $\epsilon \to 0$.
To this end, we establish the following result for multiple maxima of the field $\XGFFms$.
As in \cite[Section 3.4]{MR3509015}, this together with Theorem \ref{thm:extracting-poisson-law} implies Theorem \ref{thm:ext-process-massive-GFF}.
\begin{theorem}[Analogue to {\cite[Theorem 3.4]{MR3509015}}]
\label{thm:multiple-maxima}
Let $(A_1, \ldots, A_m)$ be a collection of disjoint, non-empty and open subsets of $\Omega$.
Then the law of $\big(\max_{A_i \cap \Omega_\epsilon} \XGFFms -m_\epsilon\big)_{i=1,\ldots, m}$ converges weakly as $\epsilon \to 0$.
\end{theorem}

The proof follows closely the ideas of \cite[Section 5]{MR3509015},
which establishes Theorem \ref{thm:multiple-maxima} for the massless GFF with Dirichlet boundary condition as a generalisation of \cite[Theorem 1.1]{MR3433630}. 
As indicated at the beginning of this section, we use the decomposition of $\Omega$ into macroscopic boxes $V_i$, $i=1,\ldots, K^2$ of sidelength $1/K$.
Recall that by Lemma \ref{lem:decomposition-gff-boxes} this induces a decomposition of the field $\XGFFms$ into a fine field $\tilde X_{s,K}^f$ and a coarse field $\tilde X_{s,K}^c$ with covariances \eqref{eq:tilde-Xf} and \eqref{eq:tilde-Xc}.
This is similar to the Gibbs-Markov decomposition in \cite{MR3433630}, and in fact, our coarse field $\tilde X_{s,K}^c$ satisfies the same covariance estimates as the coarse field in \cite{MR3433630}; see for instance \cite[Lemma 4.4]{MR4399156}.
While it is possible to run the same argument as in \cite[Section 5]{MR3509015} with the decomposition $\XGFFms \stackrel{d}{=} \tilde X_{s,K}^f + \tilde X_{s,K}^c$,
we choose to replace the field $\tilde X_{s,K}^f$ by a collection of independent massless GFFs with zero boundary condition denoted as $X_K^f$. Recall that
\begin{equation}
\cov(X_K^f) = (\Delta_\Gamma)^{-1},
\end{equation}
where $\Delta_\Gamma$ is the Laplacian on $\Omega$ with Dirichlet boundary condition on $\Gamma$.
Thanks to Lemma \ref{lem:G-concentration} the error to $\tilde X_{s,K}^f$ vanishes when taking $K\to \infty$. 
Hence, we continue to work with the field
\begin{equation}
\zeta_{s,K} = X_K^f + \tilde X_{s,K}^c.
\end{equation}
Thus, the fine field is now exactly as in \cite{MR3433630} and hence,
we may invoke many results in this reference without having to worry about different covariances.
Moreover, it is easy to see that it suffices to prove Theorem \ref{thm:ext-process-massive-GFF} with $\zeta_{s,K}$ instead of $\XGFFms$ by Lemma \ref{lem:G-concentration}.
From this point on, the proof in \cite{MR3509015} can be used without major adjustments. We omit further details.

\subsection{Size of the level sets}
\label{app:size-level-sets}
The aim of this section is to prove Proposition \ref{prop:exponential-lower-bound-level-sets} below for the size of the level sets
\begin{equation}
\Gamma_\epsilon^\GFF(\lambda) = \{ x\in \Omega_\epsilon \colon \XGFFms (x) - m_\epsilon \geq \lambda \}.
\end{equation}

\begin{proposition}[\hspace{-3.5pt} Generalisation of {\cite[Proposition 4.1]{MR3262484}}]
\label{prop:exponential-lower-bound-level-sets}
There exists $c>0$ such that for all $\lambda >1$ and all large enough $\tau>0$,
\begin{equation}
\sup_{\epsilon >0} \P(|{\Gamma_\epsilon^\GFF(\lambda)} |  > e^{\tau  \lambda}) \leq e^{-c \tau \lambda}.
\end{equation}
\end{proposition}

We follow closely the proof of \cite[Proposition 4.1]{MR3262484}, which uses two results of \cite{MR3262484}.
Since these were proved for the massless GFF with Dirichlet boundary condition,
we need to generalise them accordingly.
The proofs in \cite{MR3262484} rely on a comparison with a branching random walk (BRW) and a modified branching random walk (MBRW).
In our case it suffices to consider the modified branching random walk, introduced in \cite{MR2846636}, which is defined as follows.

Let $\epsilon=\frac{1}{2^n}$, where $n>0$ is an integer.
For $k \in \{0,\ldots, n\}$ we let $\cB_k$ be the collection of squared boxes in $\epsilon \Z^2$ of sidelength $\epsilon 2^k = \frac{1}{2^{n-k}}$ with corners in $\epsilon \Z^2$,
and let $\cB \cD_k$ denote the subset of $\cB_k$, which consists of boxes of the form $[0,2^{-{n-k}}) \cap \epsilon\Z^2 + (i 2^{-{n-k}}, j 2^{-{n-k}})$.
Note that $\cB \cD_k$ is a partition of $\epsilon\Z^2$ into boxes of sidelength $\epsilon 2^{k}$. 
For $v \in \epsilon \Z^2$, let $\cB_k(v)$ be the collection of boxes in $\cB_k$ that contain $v$ and define $\cB \cD_k(v)$ accordingly. 
Note that $\cB_k(v) = 2^{2k}$ and $\cB \cD_k(v)=1$.
Moreover, denote by $\bar \cB_k$ the subset of $\cB_k$ consisting of boxes, whose lower left corner is in $[0,1) \cap \epsilon \Z^2$.
For $k\in \{0,\ldots, n\}$ and $B\in \bar \cB_k$, let $\bar b_{k,B}$ be independent centred Gaussian random variables with 
$\var(\bar b_{k,B}) = 2^{-2k}$ and define for $B\in \cB_k$
\begin{equation}
b_{k,B} = \bar b_{k,B'} \qquad \text{for} \qquad B \sim_\epsilon B' \in \bar \cB_k,
\end{equation}
where $B \sim_\epsilon B'$ if and only if $B=(i,j) + B'$ for some $i,j \in \Z$.
Then the MBRW $\{\xi_v^\epsilon \colon v\in \Omega_\epsilon\}$ is defined by
\begin{equation}
\label{eq:mbrw-def}
\xi_v^\epsilon =  \sqrt{\frac{{\log 2}}{{2\pi}}}\sum_{k=0}^n \sum_{B\in \cB_k(v)} b_{k,B}.
\end{equation}

Compared to \cite{MR2846636} and \cite{MR3262484} we multiplied the MBRW by a suitable constant so that its variance matches the variance of $\XGFFms$ in first order.
Note that, unlike for the usual branching random walk, we have that $\xi^\epsilon$ has periodic boundary condition on $\Omega_\epsilon$.
 
For $u,v \in \epsilon \Z^2$ let $ \bar d (u,v) = \min_{v\sim_\epsilon w } | u-w|$ be the Euclidean distance between $u$ and $v$ when identifying $\Omega_\epsilon \cong \epsilon \Z^2/ \sim_\epsilon$.
Note that $\bar d$ is the Euclidean distance on the torus $\Omega$, and thus,
when $x,y \in \Omega_\epsilon$, we continue to write $\bar d(x,y) = |x-y|$. Then the covariance of the MBRW is given by Lemma \ref{lem:covariance-mbrw} below.
In light of Lemma \ref{lem:covariance-XsGFF}, we see that it approximates the covariance of the torus GFF $\XGFFms$ up to a constant depending on $m$ and $s$.

\begin{lemma}[\hspace{-3.5pt} Exactly {\cite[Lemma 2.1]{MR3262484}}]
\label{lem:covariance-mbrw}
Let $\epsilon = \frac{1}{2^n}$ and let $\{\xi_x^\epsilon \colon x \in \Omega_\epsilon \}$ be defined as in \eqref{eq:mbrw-def}. Then we have for all $x,y \in \Omega_\epsilon$ as $n\to \infty$
\begin{equation}
\cov (\xi_x^\epsilon, \xi_y^\epsilon) = - \frac{1}{2\pi} \log | x - y|  + O(1).
\end{equation}
\end{lemma}

A key notion in the proof of Proposition \ref{prop:exponential-lower-bound-level-sets} is the maximal sum over vertices for the MBRW $\xi^\epsilon$ and the torus GFF $\XGFFms$ defined by
\begin{align}
\label{eq:maximal-sum-mbrw}
\cQ_{\ell, \epsilon} &= \max \big\{  \sum_{v\in A} \xi^\epsilon (v) \colon |A|=\ell, A\subseteq \Omega_\epsilon \big\}, \\
\label{eq:maximal-sum-gff}
\cS_{\ell, \epsilon} &= \max \big\{  \sum_{v\in A} \XGFFms (v) \colon |A|=\ell, A\subseteq \Omega_\epsilon \big\}.
\end{align}
Moreover, as in \cite{MR3262484}, we denote the set of pairs of points at intermediate scales by
\begin{equation}
\Xi_{\epsilon, r} = \{ (u,v) \in \Omega_\epsilon \times \Omega_\epsilon \colon r\epsilon \leq |u-v| \leq 1/r  \}.
\end{equation}

\begin{lemma}[Analogue of {\cite[Lemma 4.6]{MR3262484}}]
\label{lem:estimate-probability-pairs-of-points}
There exist absolute constants $C,c>0$, such that for all $\epsilon>0$ and $r,\lambda >C$,
\begin{align}
\label{eq:estimate-probability-pairs-of-points}
\P (\exists A \subseteq \Xi_{\epsilon, r}  \text{ with } |A| \geq \log r \colon \forall u,v \in A \colon
\XGFFms(u) + \XGFFms(v) &\geq 2 m_\epsilon - 2\lambda \log \log r) \nnb
&\geq 1-Ce^{-e^{\lambda \log \log r}}.
\end{align}
\end{lemma}

\begin{proof}
We follow closely the proof of \cite[Lemma 4.6]{MR3262484} and use large parts of the notion therein.
Let $R= (\log r)^{-\lambda}$ and $\ell = (\log r)^{\lambda/100}$. We implicitly assume that $R$ and $\ell$ are integer multiples of $\epsilon$.
Let $o=(0,0) \in \Omega_\epsilon$ and define $o_i= (i \ell, 2R)$ for $1 \leq i \leq  M = \floor{1/2\ell} = \floor{(\log r)^{\lambda/100}/2}$.

Let $\cC_i$ be a discrete ball of radius $R$ around $o_i$ and let $B_i \subseteq \cC_i$ be a box of side length $R/8$ centred at $o_i$. 
Note that by construction the boxes are contained in the circles and the circles do not intersect.

Next, we regroup the $M$ circles and boxes into $m= (\log r )^{\lambda/200}$ blocks as follows.
For $j=1, \ldots, M/m$ define the $j$-th blocks by $\mathfrak{C}_j = \{ \cC_i \colon (j-1)m <i \leq jm \}$ and $\mathfrak{B} = \{ B_i \colon (j-1)m <i \leq jm \}$.
The goal is to find in each box a pair of points $(u,v) \in \Xi_{\epsilon, r}$ that satisfies the condition in the probability in \eqref{eq:estimate-probability-pairs-of-points}.

To this end, we consider the maximal sum over pairs of $\XGFFms$ in each block of boxes $\mathfrak{B}_j$. For concreteness, we carry out the argument for $j=1$.
To ease the notation we write $\mathfrak{B} = \mathfrak{B}_1$ and $\mathfrak{C}=\mathfrak{C}_1$.
Moreover, we denote by $\partial \cC$ the boundary of a circle $\cC$ and set $\partial \mathfrak{C} = \cup_{C \in \mathfrak{C}} \partial \cC$.

Let $\Delta_{\partial \mathfrak{C}}$ be the Laplacian with zero boundary condition on $\partial \mathfrak{C}$ defined analogously to $\Delta_\Gamma$.
Then by the same arguments as in the proof of Lemma \ref{lem:decomposition-gff-boxes} we can write
\begin{equation}
\XGFFms(v) = g_v^B + \phi_v \qquad \text{ for all } v \in B \subseteq \cC \in \mathfrak{C}, 
\end{equation}
where $g^B$ is a Gaussian free field on $\cC \supseteq B$ with mass $\sqrt{m^2 + 1/s}$ and Dirichlet boundary condition on $\partial \cC$ such that $\{ g^B \colon B \in \mathfrak{B} \}$ is a collection of independent fields, and the binding field $\phi$ is independent of everything with
\begin{equation}
\cov (\phi )= (-\Delta_{\partial \mathfrak{C}} + m^2 + 1/s)^{-1} - (-\Delta + m^2 +1/s)^{-1}.
\end{equation} 
For $B \in \mathfrak{B}$ define $( \chi_{1,B}, \chi_{2,B} ) \in B\times B \cap \Xi_{\epsilon, r}$, such that
\begin{equation}
g_{\chi_{1,B}}^B + g_{\chi_{2,B}}^B = \max_{u,v \in B\times B \cap \Xi_{\epsilon, r}} g_v^B + g_u^B.
\end{equation}

Then we have similarly as in \cite{MR3262484} for $\lambda$ large enough
\begin{equation}
\P(g_{\chi_{1,B}}^B + g_{\chi_{2,B}}^B \geq 2m_\epsilon - \lambda \log \log r) \geq 1/4.
\end{equation}

Let $W= \{ (\chi_{1,B}, \chi_{2,B}) \colon g_{\chi_{1,B}}^B + g_{\chi_{2,B}}^B \geq 2m_\epsilon - \lambda \log \log r,\, B\in \cB \}$. Then we have by independence for some constant $c>0$
\begin{equation}
\P\big( |W| \leq \frac{1}{8}m \big) \leq e^{-cm}.
\end{equation}

To conclude, we need to prove that the binding field $\phi$ satisfies $\phi_u+ \phi_v \geq 0$ for $(u,v) \in W$ with high probability.
To this end, we generalise \cite[Lemma 4.7]{MR3262484}, which itself is a generalisation of \cite[Lemma 2.3]{MR3101848}. 
This estimate can be proved by using only the covariance estimates in \cite[Lemma 2.2]{MR3101848} as remarked below equation \cite[(16)]{MR3101848}.
Since our covariance estimate only differs in the constants, which depend on $m$ and $s$,
we have for some constants $c,C>0$
\begin{equation}
\P\big (\phi_u + \phi_v \leq 0 \text{ for all } (u,v) \in W \big) \leq C e^{-c(\log r)^{c\lambda}}. 
\end{equation}

Going back the the proof of the Lemma of \cite[Lemma 4.6]{MR3262484}, we can conclude
\begin{align}
\P\big(\max_{B \in \cB} \max_{u,v \in B  \times B \cap \Xi_{\epsilon, r}} \XGFFms (u) + \XGFFms (v) \geq 2m_\epsilon -2 \lambda \log \log r \big) \geq 1- Ce^{-c (\log r)^{c\lambda} }.
\end{align}

We can apply the same argument to all $\mathfrak{B}_j$.
Thus, with probability $1- Ce^{-c (\log r)^{c\lambda} }$ there is a pair of points $(\chi_{1,j}, \chi_{2,j}) \in \Xi_{\epsilon,r}$, which satisfies
\begin{equation}
\XGFFms(\chi_{1,j}) + \XGFFms( \chi_{2,j} ) \geq 2 m_\epsilon - 2\lambda \log \log r.
\end{equation}
Now, define $A= \{ (\chi_{1,j}, \chi_{2,j}) \colon 1\leq j \leq m \}$.
Note that for $r$ large enough, we have $|A| = M/m = (\log r)^{\lambda/200}/2 \geq \log r$.
Moreover, a union bound yields for some constant $\tilde C > C$
\begin{align}
\P\big(\min_{(u,v)\in A} \XGFFms (u) + \XGFFms (v) &\leq 2m_\epsilon - 2\lambda \log \log r \big) \nnb
&\leq \sum_{j=1}^{M/m} \P(\XGFFms (\chi_{1,j}) + \XGFFms (\chi_{2,j}) \leq 2m_\epsilon - 2\lambda \log \log r ) \nnb
&\leq (\log r)^{\lambda/200}/2 Ce^{-c (\log r)^{c\lambda} } \leq \tilde C e^{-c (\log r)^{c\lambda} }, 
\end{align}
from which the estimate \eqref{eq:estimate-probability-pairs-of-points} follows.
\end{proof}

\begin{lemma}[Analogue to {\cite[eq. (59)]{MR3262484}}]
\label{lem:expectation-maximal-sum}
There is a constant $c>0$ such that for sufficiently large $\ell>0$, we have
\begin{equation}
\E \cS_{\ell,\epsilon} \leq \ell (m_\epsilon - c\log \ell).
\end{equation}
\end{lemma}

\begin{proof}
We follow closely the first part in the proof of \cite[Theorem 1.2]{MR3262484},
where the following argument was carried out for the usual BRW.
We argue that it also holds for the MBRW. Let
\begin{equation}
\Xi_{\epsilon, x}^* = \bigcup_{i=m_\epsilon - M_\epsilon}^{x} \Xi_\epsilon(i),
\end{equation}
where $M_\epsilon$ is the maximum of the MBRW $\xi^\epsilon$ and
\begin{equation}
\Xi_\epsilon(x) = \big \{ x\in \Omega_\epsilon \colon \xi^\epsilon (x) \in [m_\epsilon-x-1, m_\epsilon-x] \big\}
\end{equation}
is the set of points in $\Omega_\epsilon$ at distance roughly $x$ behind the maximum.
Then we have for the maximal sum $\cQ_{\ell, \epsilon}$ defined in \eqref{eq:maximal-sum-mbrw} 
\begin{align}
\cQ_{\ell, \epsilon} 
&\leq \ell M_\epsilon \mathbf{1}_{\{  | \Xi_{\epsilon, \log \ell /2\alpha} | > \ell/2  \}} 
+(\ell M_\epsilon - \ell \log \ell/4\alpha )\mathbf{1}_{\{  |\Xi_{\epsilon, \log \ell / 4\alpha }| \leq \ell/2  \}}
\nnb
&\leq \ell M_\epsilon - \mathbf{1}_{\{|\Xi_{\epsilon, \log \ell /2\alpha }| \leq \ell/2 \} } \ell \log \ell / 4\alpha,
\end{align}
where $\alpha = \sqrt{8\pi}$. 
Applying expectation, we may invoke \cite[Lemma 3.7]{MR3262484} and a generalisation of \cite[Proposition 3.3]{MR3262484} to the MBRW,
from which the statement follows with $\cQ_{\ell, \epsilon}$ in place of $\cS_{\ell, \epsilon}$ for a suitable constant $c>0$.
Using the Sudakov-Fernique inequality as stated in \cite[Lemma 2.2]{MR3262484} togehter with \cite[Lemma 2.7]{MR3262484} and the ideas in the proof of \cite[Lemma 2.9]{MR3262484}, it follows that
\begin{equation}
\E[\cS_{\ell,\epsilon}] \leq \E [\cQ_{\ell, 2^{-k}\epsilon} ]
\end{equation}
for some $k\in \N$.
This concludes the proof of Lemma \ref{lem:expectation-maximal-sum}.
\end{proof}

\begin{proof}[Proof of Proposition \ref{prop:exponential-lower-bound-level-sets}]
The proof in \cite{MR3262484} is based on \cite[Lemma 4.2]{MR3262484} and \cite[Lemma 4.3]{MR3262484}.
To prove the analogues of these results for the torus GFF $\XGFFms$ we use the same arguments as in \cite{MR3262484},
noting that all covariance dependencies are isolated in Lemma \ref{lem:estimate-probability-pairs-of-points} and Lemma \ref{lem:expectation-maximal-sum}.

The rest of the argument in \cite{MR3262484} can then be used without further adjustment.
This concludes the proof of Proposition \ref{prop:exponential-lower-bound-level-sets}.
\end{proof}

\section*{Acknowledgements}
The author thanks Roland Bauerschmidt for the initial discussion of this topic and helpful comments on the first draft of this paper.
The author also thanks Marek Biskup and Oren Louidor for essential comments that led to the approach of considering the joint convergence of the extremal process and the continuous field in Theorem \ref{thm:convergence-laplace-functionals-phis},
as well as Ofer Zeitouni for the discussion on the generalisation of Theorem \ref{thm:ext-process-massive-GFF}.
Finally, the author thanks the two referees for remarks and suggestions on the first draft, which significantly improved the quality of the exposition of the content.
This work was partially supported by the UK EPSRC grant EP/L016516/1 for the Cambridge Centre for Analysis.

\bibliography{all}
\bibliographystyle{abbrv}
\end{document}